\documentclass{article}

\usepackage[french,english]{babel}
\usepackage[utf8]{inputenc}
\usepackage[T1]{fontenc}
\usepackage{lmodern}
\usepackage{amsmath}
\usepackage{amssymb}
\usepackage{mathrsfs}
\usepackage{amsthm}
\usepackage[shortlabels]{enumitem}
\usepackage{dsfont} 
\usepackage{eurosym} 
\usepackage{stmaryrd} 

\usepackage{titlesec}
\titleclass{\part}{top}
\titleformat{\part}[display]
{\huge\bfseries\centering}{\partname~\thepart}{0pt}{}
\titlespacing*{\part}{0pt}{160pt}{40pt}

\usepackage{tikz}
\usetikzlibrary{arrows, decorations.pathreplacing}

\usepackage{graphicx}

\usepackage{float}

\usepackage{hyperref}  
\hypersetup{                    
	colorlinks=true,                
	breaklinks=true,                
	urlcolor= blue,                 
	linkcolor= black,                
	citecolor= magenta,                
	pdfstartview = FitH,	bookmarksopen = true               
}

\usepackage{listings}
\usepackage{xcolor}
\usepackage{framed}

\usepackage{fullpage} 

\newcommand{\Jast}{{\mathcal J}}
\newcommand{\AW}{\mathcal{AW}}
\newcommand{\W}{\mathcal{W}}
\newcommand{\Leb}{{\rm Leb}}
\newcommand{\SC}{{SC}}
\newcommand{\fp}[1]{{\mathbf #1}}

\newcommand{\X}{\mathcal{X}}
\newcommand{\Y}{\mathcal{Y}}
\newcommand{\U}{\mathcal{U}}
\newcommand{\mean}{\operatorname{mean}}
\newcommand{\price}{\operatorname{price}}
\newcommand{\law}{\operatorname{law}}

\newcommand{\R}{\mathbb{R}}
\newcommand{\N}{\mathbb{N}}

\DeclareMathOperator{\id}{Id}
\DeclareMathOperator{\supp}{Supp}

\DeclareMathOperator{\proj}{proj}

\DeclareMathOperator{\co}{\operatorname{co}}

\renewcommand{\epsilon}{\varepsilon}

\newcommand{\1}{\mathds{1}}

\newcommand{\mathbbm}[1]{\1}

\title{An extension of martingale transport and stability in robust finance}
\author{Benjamin Jourdain\thanks{CERMICS, Ecole des Ponts, INRIA, Marne-la-Vallée, France. E-mails: benjamin.jourdain@enpc.fr} \and Gudmund Pammer\thanks{ETH Z\"urich, Switzerland. E-mail: gudmund.pammer@math.ethz.ch}}
\date{\today}

\numberwithin{equation}{section}
\theoremstyle{plain}
\newtheorem{prooff}{Proof}[section]

\addto\captionsfrench{}

\newtheorem{lemma}[prooff]{Lemma}

\newtheorem{theorem}[prooff]{Theorem}
\newtheorem{proposition}[prooff]{Proposition}
\newtheorem{corollary}[prooff]{Corollary}

\newtheorem{assumpA}{Assumption}
\newtheorem{assumpB}{Assumption}

\theoremstyle{definition}

\newtheorem{remark}[prooff]{Remark}

\theoremstyle{definition}

\theoremstyle{plain}
\newtheorem{demo}[prooff]{Démonstration}

\allowdisplaybreaks

\begin{document}
\maketitle
\begin{abstract} 
While many questions in robust finance can be posed in the martingale optimal transport framework or its weak extension, others like the subreplication price of VIX futures, the robust pricing of American options or the construction of shadow couplings necessitate additional information to be incorporated into the optimization problem beyond that of the underlying asset.
In the present paper, we take into account this extra information by introducing an additional parameter to the weak martingale optimal transport problem. 
We prove the stability of the resulting problem with respect to the risk neutral marginal distributions of the underlying asset, thus extending the results in \cite{BeJoMaPa21b}. 
A key step is the generalization of the main result in \cite{BJMP22} to include the extra parameter into the setting.
This result establishes that any martingale coupling can be approximated by a sequence of martingale couplings with specified marginals, provided that the marginals of this sequence converge to those of the original coupling.
Finally, we deduce stability of the three previously mentioned motivating examples.
\end{abstract}

{\bf Keywords:} Martingale Optimal Transport, Adapted Wasserstein distance, Robust finance, Weak transport, Stability, Convex order, Martingale couplings.
\section{Introduction}
In mathematical finance, the evolution of an asset price on a financial market is modeled by an adapted stochastic process $(X_t)$ on a filtered probability spaces $(\Omega,\mathcal F,\mathbb P, (\mathcal F_t))$.
To ensure the absence of arbitrage opportunities, risk-neutral measures (also known as equivalent martingale measures) $\mathbb Q$ are considered under which the asset price process $(X_t)$ is a martingale, up to assuming zero interest rates.
The reason why a transport type problem arises in robust finance is because the marginals of $(X_t)$ can be derived from market information based on the celebrated observation of Breeden--Litzenberger \cite{BrLi78}.
According to this observation, the prices of traded vanilla options determine the marginals $(\mu_t)$ of $(X_t)$ at their respective maturity times under the risk-neutral measure $\mathbb Q$.
Instead of considering one specific financial model, a robust approach is to consider all martingale measures that are compatible with this observation, that is, all filtered probability spaces $(\Omega,\mathcal F,\mathbb Q, (\mathcal F_t))$ and stochastic processes $(X_t)$ such that
\begin{equation}
    \label{intro:compatible}
    X\text{ is a }(\mathbb Q,(\mathcal F_t))\text{-martingale and } X_t \sim \mu_t \text{ at all maturity times $t$}.
\end{equation}
Then the robust price bounds for an option with payoff $\Phi$ are obtained by solving a transport type problem \cite{BeHePe12,GaHeTo13} where the optimization takes place over the set of all risk-neutral measures that are compatible with the observed prices of vanilla options.
That are martingale measures $\mathbb Q$ under which $(X_t)$ has the correct marginal distributions, i.e.,
\begin{equation}\label{intro:MMOT}
    \inf / \sup \left\{  \mathbb E_\mathbb Q[\Phi] : (\Omega,\mathcal F,\mathbb Q,(\mathcal F_t),(X_t)) \text{ satisfying }\eqref{intro:compatible} \right\}.
\end{equation}
However, as we can only observe the prices of a finite number of derivatives (up to a bid ask spread), the marginals $(\mu_t)$ are merely approximately known.
Therefore, it is crucial to establish the stability of the transport type problem \eqref{intro:MMOT} with respect to the marginals.

This article is concerned with the one time period setting, that is $t \in \{1,2\}$
.
Then, when $\Phi$ is written on the underlying asset $X$, \eqref{intro:MMOT} boils down to a martingale optimal transport (MOT) problem
\begin{equation}
    \label{intro:MOT}
    \underset{\pi \in \Pi_M(\mu_1,\mu_2)}{\inf / \sup}  \int \Phi(x,y) \, \pi(dx,dy),
\end{equation}
where $\Pi_M(\mu_1,\mu_2)$ denotes the set of martingale couplings with marginals $\mu_1$ and $\mu_2$, i.e., the set of laws of 1-time step martingales $(X_1,X_2)$ with $X_t \sim \mu_t$.
Continuity of the value of \eqref{intro:MOT} w.r.t.\ the marginal input, which is called stability, has been proved in \cite{BaPa19, Wi19}.
Weak martingale optimal transport (WMOT) is a nonlinear generalization of MOT analogous to weak optimal transport, which is a nonlinear generalization of classical optimal transport proposed by Gozlan, Roberto, Samson and Tetali \cite{GoRoSaTe17}, and was considered in \cite{BaPa19,BeJoMaPa21b}.
In WMOT one allows for more general payoffs $\Phi$ which may depend on the conditional law of $X_2$ given $X_1$ in addition to $X$ itself, and the corresponding WMOT problem reads as
\begin{equation}
    \label{intro:WMOT}
    \underset{\pi \in \Pi_M(\mu_1,\mu_2)}{\inf / \sup} \int \Phi\left(x, {\pi_x}\right) \, \mu_1(dx),
  \end{equation}
{where $\pi_x$ comes from the desintegration $\pi(dx,dy)=\mu_1(dx)\pi_x(dy)$.}
Stability of WMOT has been studied in \cite{BeJoMaPa21b} and was therein used to establish stability of the superreplication price of VIX futures and the stretched Brownian motion. 

Even though many problems in robust finance are covered by 
WMOT, some important examples require that information is included into the optimization problem beyond that of the underlying asset.
Accordingly these problems can not be properly treated in the 
WMOT frameworks.
For us, guiding examples of such problems are the subreplication price of VIX futures, the robust pricing of American options and the construction of shadow couplings.
Through augmenting WMOT by an additional parameter, we demonstrate how this extra information can be taken into account, prove stability of the resulting problem, and consequently deduce stability of the three guiding examples. 
A key step is the generalization of the main result in \cite{BJMP22} to our current setting.
This result states that any martingale coupling can be approximated by a sequence of martingale couplings with specified marginals, provided that the marginals of this sequence converge to those of the original coupling.
As a side product of our approach, we establish the very same result on the level of stochastic processes with general filtrations (c.f.\ \cite{BaBePa21}) any 1-step martingale on some filtered probability space can be approximated w.r.t.\ the adapted Wasserstein distance by martingales on (perhaps different) filtered probability spaces, provided that the marginals of this sequence converge to those of the original martingale.

\subsection{Notation}

Let $(\X,d_\X)$ and $(\Y, d_\Y)$ be Polish metric spaces and $p \ge 1$
We equip the product $\X \times \Y$ with the product metric $d_{\X \times \Y}((x,y),(\tilde x,\tilde y)) := (d_\X(x,\tilde x)^p + d_\Y(y,\tilde y)^p)^{1 / p}$ which turns $\X \times \Y$ into a Polish metric space.
The set of Borel probability measures on $\X$ is denoted by $\mathcal P(\X)$.
For $\mu \in \mathcal P(\X)$ and $\nu \in \mathcal P(\Y)$, we write $\Pi(\mu,\nu)$ for the set of all probability measures on $\X \times \Y$ with marginals $\mu$ and $\nu$.
We denote by $\mathcal P_p(\X)$ the subset of $\mathcal P(\X)$ that finitely integrates $x \mapsto d_\X^p(x,x_0)$ for some (thus any) $x_0 \in \X$ and endow $\mathcal P_p(\X)$ with the $p$-Wasserstein distance $\mathcal W_p$ so that $(\mathcal P_p(\X), \mathcal W_p)$ is a Polish metric space where, for $\mu,\nu \in \mathcal P_p(\X)$,
\begin{equation}
    \label{eq:defwp}
    \mathcal W_p(\mu, \nu) := 
    \left( \inf_{\pi \in \Pi(\mu, \nu)} \int d_\X(x, y)^p \, \pi(dx,dy) \right)^{1 / p}.
\end{equation} 
The set of continuous and bounded functions on $\X$ is denoted by $C_b(\X)$ and we use the shorthand notation $\mu(f)$ to write the integral of a $\mu$-integrable function $f \colon \X \to \R \cup \{ \pm \infty \}$ w.r.t.\ a Borel measure $\mu$ on $\X$.
Given a measurable map $f \colon \X \to \Y$, we denote by $f_\# \mu$ the push-forward measure of $\mu$ under $f$.
For Polish spaces $\X_1,\X_2,\X_3$ and $\pi \in \mathcal P(\X_1 \times \X_2 \times \X_3)$ and a non-empty subset $I$ of $\{1,2,3\}$, $\proj_I \pi$ denotes the image of $\pi$ by the projection to the coordinates in $I$, for example, $\proj_1 \pi$ is the $\X_1$-marginal of $\pi$.
Further, we write $\pi_{x_1,x_2}$ for the disintegration of $\pi(dx_1,dx_2,dx_3)=\proj_{1,2}\pi(dx_1,dx_2) \pi_{x_1,x_2}(dx_3)$.
Frequently, we use the injection (c.f.\ \cite[Section 2]{BaBePa18})
\begin{align*}
    J \colon& \mathcal P( \X_1 \times \X_2 \times \X_3) \to \mathcal P ( \X_1 \times \X_2 \times \mathcal P(\X_3) ) \\
    & \pi \mapsto \left((x_1,x_2,x_3) \mapsto (x_1, x_2, \pi_{x_1,x_2})\right)_\# \pi,
\end{align*}
and remark that $J(\pi^k) \to J(\pi)$ in $\mathcal P_p(\X_1 \times \X_2 \times \mathcal P_p(\X_3))$ implies $\pi^k \to \pi$ in $\mathcal P_p(\X_1 \times \X_2 \times \X_3)$.

Unless stated otherwise, $\mathbb R$ is equipped with the Euclidean distance and $\Leb$ denotes the Lebesgue measure on $[0,1]$.
Two measures $\mu,\nu \in \mathcal P_1(\R)$ are said to be in the convex order and we write $\mu\le_{cx}\nu$, if 
\[ 
    \forall \varphi\colon \R\to\R\mbox{ convex,}\quad \mu(\varphi) \le \nu(\varphi).
\]
We write $\mean \colon \mathcal P_1(\R) \to \R$ for $\mean(\rho) = \int y \, \rho(dy)$ and denote by
\[
    \Lambda_M(\mu,\nu) := \left\{ P \in \mathcal P_1(\R \times \U \times \mathcal P_1(\R)) : \int \delta_x \otimes \rho \, P(dx,d\rho) \in \Pi(\mu,\nu), \, \mean(\rho) = x \, P(dx,d\rho)\text{-a.s.} \right\}.
\]

\subsection{Organization of the paper}
Section \ref{sec:main} presents the main results of this paper.
First, we introduce in Subsection \ref{ssec:extension} the setup with the additional parameter and state in Theorem \ref{thm:intro.approximation} and Theorem \ref{thm:intro.stability} the corresponding results related to stability.
Furthermore, we present in Subsection \ref{ssec:main.process} consequences of these results in the filtered process setting, namely Corollary \ref{cor:filpro}.
Subsequently, we explain and state stability of the three guiding examples, that are, subreplication of VIX futures (Subsection \ref{ssec:main.VIX}), robust pricing of American options (Subsection \ref{ssec:main.American}), and shadow couplings (Subsection \ref{ssec:main.SC}).
Section \ref{sec:proofs} is concerned with the proofs.

\section{Main results}\label{sec:main}

\subsection{An extension of martingale transport} \label{ssec:extension}
We introduce now a framework that is sufficiently general to deal with the question of stability of our guiding examples.
From now on, let $(\U,d_{\U})$ be a Polish metric space that models an extra information parameter $u \in \U$.
Given $\bar \mu \in \mathcal P_1(\R \times \U)$ and $\nu \in \mathcal P_1(\R)$ with $\proj_1 \bar \mu \le_{cx} \nu$, we denote by $\Pi_M(\bar \mu,\nu)$ the set of couplings $\pi \in \Pi(\bar \mu,\nu)$ such that $\mean(\pi_{x,u}) = x$ $\bar\mu(dx,du)$-a.e.
Central to establishing the upper (resp. lower) semicontinuity property in our stability results for minimization (resp. maximization) problems 
 is Theorem \ref{thm:approximation}
, which is a reinforced version of the result below:
\begin{theorem} \label{thm:intro.approximation}
    Let $(\bar\mu^k, \nu^k)_{k \in \N}$, $\proj_1 \bar \mu^k \le_{cx} \nu^k$, be a convergent sequence in $\mathcal P_1(\mathbb R \times \U) \times \mathcal P_1(\R)$ with limit $(\bar \mu,\nu)$.
    Then, every coupling $\pi \in \Pi_M(\bar \mu,\nu)$ is the weak limit of a sequence $(\pi^k)_{k \in \N}$ with $\pi^k \in \Pi_M(\bar \mu^k,\nu^k)$ {and $J(\pi)$ is the weak limit of $(J(\pi^k))_{k \in \N}$}.
\end{theorem}
  
In view of the counter-example by Br\"uckerhoff and Juillet \cite{BrJu22}, this result does not generalize to higher dimensions i.e.\ when $\R$ is replaced by $\R^d$ with $d\ge 2$.
This generalization of the main result of \cite{BJMP22} to the present framework is also key to establish  
the stability w.r.t.\ the marginals of the following variant of WMOT:
\begin{equation}
    \label{eq:pwmot}
    V_C(\bar\mu,\nu) := \inf_{\bar \pi \in \Pi_M(\bar \mu,\nu)}
    \int C(x,u,\bar \pi_{x,u}) \, \bar \mu(dx,du).
\end{equation}
As usual, it is necessary to impose regularity on the cost $C$ in order to have a continuous dependence of the optimal value of \eqref{eq:pwmot} w.r.t.\ the marginals.
Thus, we will suppose the following continuity assumption on the cost function:
\begin{assumpA}
    \label{ass:cost.A}
    We say $C \colon \mathbb R \times \U \times \mathcal P_p(\mathbb R) \to \mathbb R$ satisfies Assumption \ref{ass:cost.A} if
    $C$ is continuous and there is $K > 0$ such that, for all $(x,u,\rho) \in \R \times \U \times \mathcal P_p(\R)$ and some $u_0 \in \U$,
    \begin{equation}
        \label{eq:ass.cost.bound}
        |C(x,u,\rho)| \le K\left( 1 + |x|^p + d_\U^p(u,u_0) + \int |y|^p \, \rho(dy) \right).
    \end{equation}
  \end{assumpA}
  
\begin{theorem} \label{thm:intro.stability}
    Let $C$ satisfy Assumption \ref{ass:cost.A} and $C(x,u,\cdot)$ be convex for all $(x,u)\in\R\times\U$. Then the value function $V_C$ is attained and continuous on $\{ (\bar \mu,\nu) : \proj_1 \bar \mu \le_{cx} \nu \} \subseteq \mathcal P_p(\R \times \U) \times \mathcal P_p(\R)$.
    Furthermore, when $(\bar \mu^k, \nu^k)_{k \in \N}$, $\proj_1 \bar \mu^k \le_{cx} \nu^k$, converges to $(\bar \mu, \nu)$, we have: 
    \begin{enumerate}[label = (\roman*)]
        \item if, for $k \in \N$, $\pi^k \in \Pi_M(\bar \mu^k,\nu^k)$ is optimal for \eqref{eq:pwmot}, so are accumulation points of $(\pi^k)_{k \in \N}$;
        \item if additionally $C(x,u,\cdot)$ is strictly convex, then optimizers to \eqref{eq:pwmot} are unique.
        {Furthermore, $(\pi^k)_{k \in \N}$ and $(J(\pi^k))_{k \in \N}$ weakly converge to the optimizer of \eqref{eq:pwmot} with marginals $(\bar \mu,\nu)$ and its image under $J$, respectively.}
    \end{enumerate}
  \end{theorem}

\subsection{VIX futures} \label{ssec:main.VIX}
The VIX is the implied volatility of the $30$-day variance swap on the S\&P 500. 
According to Guyon, Menegaux and Nutz \cite{GuMeNu17}, the subreplication price at time $0$ for the VIX future expiring at $T_1$ is given by
\begin{equation}
    \label{eq:VIX.sub.primal}
    P_{\rm sub}(\mu,\nu) = \sup \{ \mu(\phi) + \nu(\psi) \},
\end{equation}
where $\mu$ and $\nu$ denote the risk neutral distributions of the S\&P 500 at dates $T_1$ and $T_2$ equal to $T_1$ plus $30$ days both inferred from the market prices of liquid options. Moreover, the supremum is taken over all $(\phi,\psi) \in L^1(\mu) \times L^1(\nu)$ and measurable maps $\Delta^S, \Delta^L$ such that, for all $(x,u,y) \in (0,\infty)\times [0,\infty)\times(0,\infty)$,
\begin{equation}
    \label{eq:VIX.sub.constraint}
    \phi(x) + \psi(y) + \Delta^S(x,v) (y - x) + \Delta^L(x,u) \left(\ell_x(y) - u^2 \right) - u \le 0,
  \end{equation}
  with $\ell_x(y) := \frac{2}{T_2 - T_1}\ln(x/y)$. Up to assuming zero interest rates, the S\&P 500 is a martingale under the risk neutral measure so that both, $\mu$ and $\nu$, have finite first moments and $\mu$ is smaller than $\nu$ in the convex order.
To state the dual problem, we define the set $\Pi_{\rm VIX}(\mu,\nu)$ of admissible martingale couplings as
\[
  \{ \pi \in \mathcal P((0,\infty)\times [0,\infty)\times(0,\infty)) : \proj_{1} \pi = \mu, \, \proj_{3} \pi = \nu,\, \pi_{x,u}( \id,\ell_x) = (x,u^2) \text{ for } \proj_{1,2}\pi\text{-a.e.\ }(x,u) \},
\]
with $\id$ the identity function on $\R$.
Note that each $\pi\in\Pi_{\rm VIX}(\mu,\nu)$ satisfies $\pi\in\Pi_M(\proj_{1,2}\pi,\proj_3\pi)$ and we have, by concavity of the logarithm function and Jensen's inequality, for $\proj_{1,2}\pi$-a.e. $(x,u)$ that $\pi_{x,u}(\ell_x)\ge 0$.
Given probability measures $\mu,\nu$ on $(0,\infty)$ that are in the convex order and finitely integrate $|\ln(x)| + |x|$, the dual problem $D_{\rm sub}$ consists of
    \begin{equation}
        \label{eq:VIX.sub.dual}
        D_{\rm sub}(\mu,\nu) = \inf_{ \Pi_{\rm VIX}(\mu,\nu) } \int\sqrt{\pi_{x,u}(\ell_x)} \, \proj_{1,2} \pi(dx,du).
      \end{equation}
According to \cite[Theorem 4.1]{GuMeNu17}, the values of $P_{\rm sub}(\mu,\nu)$ and $D_{\rm sub}(\mu,\nu)$ coincide. In the present paper, we are going to establish the following stability result with respect to the risk-neutral marginal distributions $\mu$ and $\nu$ of the S\&P 500 at dates $T_1$ and $T_2$. 
\begin{proposition} \label{prop:stability.VIX}
    Let $(\mu^k,\nu^k)_{k \in \N}$, $\mu^k\le_{cx} \nu^k$, be a sequence in $\mathcal P((0,\infty)) \times \mathcal P((0,\infty))$ that converges weakly to $(\mu,\nu) \in \mathcal P((0,\infty)) \times \mathcal P((0,\infty))$.
    If $\lim_{k \to \infty} \int  \left(|\ln(x)| + |x|\right) \, \nu^k(dx) = \int \left(| \ln(x) | + |x|\right) \, \nu(dx)$, then
    \[
        \lim_{k \to \infty} D_{\rm sub}(\mu^k,\nu^k) = D_{\rm sub}(\mu,\nu).
    \]
  \end{proposition}
  The analogous stability result for the VIX future superreplication price is stated in \cite[Theorem 1.3]{BeJoMaPa21b} and relies on the reduction of its dual formulation to the value function of a WMOT problem, see \cite[Proposition 4.10]{GuMeNu17}.
  Such a reduction step is, in general, not possible for the dual formulation of the subreplication price and we remark that with the approach in this paper, one can recover \cite[Theorem 1.3]{BeJoMaPa21b} without recasting the problem as a WMOT problem.

  \subsection{Filtered processes} \label{ssec:main.process}
As explained in the introduction, in the robust approach it is natural to consider all martingales that are compatible with market observations.
  For this reason, we follow the approach in \cite{BaBePa21}, and call in our setting a 5-tuple
  \[
      \fp X = \left( \Omega, \mathcal F, \mathbb P, (\mathcal F_t)_{t = 1}^2, X = (X_t)_{t = 1}^2 \right),
  \]
  consisting of a filtered probability space $(\Omega, \mathcal F, \mathbb P, (\mathcal F_t)_{t = 1}^2)$ and an $(\mathcal F_t)$-adapted process $X$, a filtered process.
  We say that a filtered process $\fp X$ is a martingale if $X$ is a $(\mathcal F_t)$-martingale under $\mathbb P$. When ${\cal F}_1$ is larger than the $\sigma$-field generated by $X_1$, the conditional distributions $\law( X_2 | \mathcal F_1 )$ and $\law(X_2|X_1)$ may differ and then $\law( X_2 | \mathcal F_1 )$ is not determined by the law of $X$. 
  For $\mu,\nu \in \mathcal P_p(\R)$ with $\mu\le_{cx}\nu$, we write $\mathbf M(\mu,\nu)$ for the set of all martingales $\fp X$ with $X_1 \sim \mu$ and $X_2 \sim \nu$.

  Assume that the payoff $\Phi$ of the option, which we want to price, depends on the conditional law $\law( X_2 | \mathcal F_1)$ in addition to the price of the asset $X$: $\Phi$ is a measurable function with domain $\R \times \R\times \mathcal P_p(\R) $ satisfying $|\Phi(x_1,x_2,\rho)|\le K \left( 1 + |x_1|^p+ |x_2|^p+ \int_\R|y|^p\, \rho(d y) \right)$ for some $K>0$. The robust price bound is given by
 \begin{equation}V_\Phi(\mu, \nu)=\sup_{\fp X \in \mathbf M(\mu,\nu)} \mathbb E[\Phi(X,\law(X_2 | \mathcal F_1))].\label{eq:WMOT.process}\end{equation}
For $C \colon \R \times \mathcal P_p(\R)\ni(x,\rho) \mapsto \int\Phi(x,y,\rho)\, \rho(dy)\in \R$, one has
  \begin{equation*}
      \label{eq:Phi.to.C}
      \mathbb E[\Phi(X, \law(X_2 | \mathcal F_1)) | \mathcal F_1] = C(X_1, \law(X_2|\mathcal F_1)).
    \end{equation*}
We connect the robust price bounds to the optimization problem
\begin{align}
      \label{eq:WMOT}
       \hat V_C(\mu,\nu) := \sup_{ P \in \Lambda_M(\mu,\nu) } P(C).
 \end{align}
Indeed, the values of \eqref{eq:Phi.to.C} and \eqref{eq:WMOT} coincide:
For $\fp X \in \mathbf M(\mu,\nu)$, we have $\law(X_1,\law(X_2|\mathcal F_1))\in\Lambda_M(\mu,\nu)$ and therefore $V_\Phi(\mu, \nu)\le \hat V_C(\mu,\nu)$.
On the other hand, for $P\in\Lambda_M(\mu,\nu)$, let us endow $\R\times{\cal P}_p(\R)\times [0,1]$ with $P\otimes\Leb$ and set $(X_1,X_2) \colon \R\times{\cal P}_p(\R)\times [0,1]\ni(x,\rho,u)\mapsto (x,F_\rho^{-1}(u))\in\R^2$. Denoting by ${\cal F}_1$ and ${\cal F}_2$ the respective Borel sigma-field on $\R\times{\cal P}_p(\R)$ and $\R\times{\cal P}_p(\R)\times [0,1]$, we have that
\[ 
    (\R\times{\cal P}_p(\R)\times [0,1],P\otimes\Leb,{\cal F}_2,({\cal F}_t)_{t=1}^2,(X_t)_{t = 1}^2) \in \mathbf M(\mu,\nu).
\]
Therefore, we also find $V_\Phi(\mu, \nu) \ge \hat V_C(\mu,\nu)$.

In the current setting, we derive the following analogue to Theorem \ref{thm:intro.approximation}.
  
  \begin{corollary}
      \label{cor:intro.approximation.process}
      Let $(\mu^k,\nu^k)_{k \in \N}$, $\mu^k \le_{cx} \nu^k$, be a convergent sequence in $\mathcal P_p(\R) \times \mathcal P_p(\R)$ with limit $(\mu,\nu)$.
      Then, every $P\in\Lambda_M(\mu,\nu)$ is the $\W_p$-limit of a sequence $(P^k)_{k \in \N}$ with $P^k\in\Lambda_M(\mu^k,\nu^k)$.
    \end{corollary}
 \begin{remark} \label{rem:intro.approximation}
   The adapted Wasserstein distance between two filtered processes $\fp X$ and $\fp Y$ is, by \cite[Theorem 3.10]{BaBePa21}, given by
  \begin{align*}
      \mathcal{AW}_p(\fp X, \fp Y) = \mathcal W_p( \law( X_1, \law( X_2 | \mathcal F_1^\fp X)) , \law( Y_1, \law( Y_2 | \mathcal F_1^\fp Y)) ).
  \end{align*}
  Consequently, the map $\fp X \mapsto \law( X_1, \law(X_2 | \mathcal F_1))$ is a surjective isometry from $\mathbf M(\mu,\nu)$ onto $\Lambda_M(\mu,\nu)$. 
  Therefore, we may rephrase Corollary \ref{cor:intro.approximation.process} using $\mathcal{AW}_p$, and obtain under the same assumptions that every process $\fp X\in \mathbf M(\mu,\nu)$ is the $\mathcal{AW}_p$-limit of a sequence of processes $(\fp X^k)_{k\in\N}$ with $\fp X^k\in \mathbf M(\mu^k,\nu^k)$.
  \end{remark} 
  Similar to Theorem \ref{thm:intro.stability} we get stability of \eqref{eq:WMOT}.
  \begin{proposition}
      \label{prop:intro.stability.process}
      Let $C \colon \R \times \mathcal P_p(\R) \to \R$ be continuous and assume that there is a constant $K > 0$ such that, for all $(x,\rho) \in \R \times \mathcal P_p(\R)$,
      \[
          |C(x,\rho)| \le K \left( 1 + |x|^p+ \int_\R|y|^p\rho(d y) \right).
      \]
      Then the value $\hat V_C$ is attained and continuous on $\{ (\mu,\nu) \in \mathcal P_p(\R)\times\mathcal P_p(\R) : \mu \le_{cx} \nu \}$.
    Moreover, if $(\mu^k,\nu^k)_{k\in\N}$, $\mu^k\le_{cx}\nu^k$ is convergent in $\mathcal P_p(\R)\times\mathcal P_p(\R)$ and $(P^k)_{k \in \N}$, $P^k \in \Lambda_M(\mu^k,\nu^k)$ is a sequence of optimizers of \eqref{eq:WMOT}, then so are its accumulation points.\end{proposition}
    As in Remark \ref{rem:intro.approximation}, it is possible to phrase Proposition \ref{prop:intro.stability.process} in the language of filtered processes.
    Since the map $\R \times \mathcal P_p(\R)\ni(x,\rho) \mapsto \delta_{x} \otimes \rho \otimes \delta_\rho \in{\cal P}_p(\R\times \R\times{\cal P}_p(\R))$ is continuous, adequate continuity and growth assumptions on $\Phi$ will imply that $C(x,\rho) := \delta_{x} \otimes \rho \otimes \delta_\rho(\Phi)$ satisfies the assumptions of Proposition \ref{prop:intro.stability.process}.
    Hence, we can deduce the following stability result for \eqref{eq:WMOT.process}.
\begin{corollary}\label{cor:filpro}
   Let $\Phi \colon \R \times \R \times \mathcal P_p(\R) \to \R$ be continuous and assume that there is a constant $K > 0$ such that, for all $(x_1,x_2,\rho) \in \R \times\R \times  \mathcal P_p(\R)$,
      \[
          |\Phi(x_1,x_2,\rho)|\le K \left( 1 + |x_1|^p+ |x_2|^p+ \int_\R|y|^p\rho(d y) \right).
        \]
    Then the value $V_\Phi$ is attained and continuous on $\{ (\mu,\nu) \in \mathcal P_p(\R)\times\mathcal P_p(\R) : \mu \le_{cx} \nu \}$.   
\end{corollary}
\subsection{American options}\label{ssec:main.American}
      The robust pricing problem of American options as considered by Hobson and Norgilas \cite{HoNo19b}, can be cast in the setting of Subsection \ref{ssec:main.process}.
      Given a filtered process $\fp X$, the filtration $(\mathcal F_t)$ models the information that is available to the buyer, who may exercise at only two possible dates, $t\in\{1,2\}$. 
      For $t\in\{1,2\}$, let $\Phi_t \colon \R^t \to \R$ be a path-dependent payoff that she receives when exercising at time $t$. 
      The model-independent price of this American option is given by
  \begin{equation}
      \label{eq:opt.stop}
      {\cal A}m(\mu,\nu)=\sup_{\fp X \in \mathbf M(\mu,\nu)} \price(\Phi; \fp X).
  \end{equation}
  As the buyer can exercise the option at any (stopping) time, the price crucially depends on the information that is available to the buyer and we have that the price of $\Phi$ is given by
  \begin{equation}
      \label{eq:opt.stop.price}
      \price(\Phi;\fp X) := \sup_{\tau \, \mathcal (\mathcal F_t)\text{-stopping time} } \mathbb E_\mathbb P\left[ \mathbbm 1_{\{ \tau = 1 \}} \Phi_1(X_1) + \mathbbm 1_{\{ \tau = 2\}} \Phi_2(X) \right].
  \end{equation}
  In the case of a Put, that is ($\Phi_1(x)=(K_1-x)^+$ and $\Phi_2(x,y)=(K_2-y)^+$),  Hobson and Norgilas \cite{HoNo19b} relate the above suprema to the left-curtain martingale coupling \cite{BeJu16} when $\mu$ does not weight points. 
  By the Snell-envelope theorem, we have that
  \[
      \price(\Phi; \fp X) = \mathbb E\left[ \max \left( \Phi_1(x), \mathbb E[ \Phi_2(x,y) | \mathcal F_1] \,  \right) \right],
  \]
  which allows us to apply here Proposition \ref{prop:intro.stability.process} with $C(x,\rho) := \max(\Phi_1(x), \int \Phi_2(x,y) \, \rho(dy))$, and deduce the following stability result:
  \begin{corollary} \label{prop:contam}
     Let $\Phi_1$ and $\Phi_2$ be continuous and $\sup_{(x,y)\in\R^2}\left(\frac{\Phi_1(x)}{1+|x|^p}+\frac{\Phi_2(x,y)}{1+|x|^p+|y|^p}\right)<\infty$.
     Then the model independent price ${\cal A}m$ is continuous on $\{ (\mu,\nu) \in \mathcal P_p(\R)^2 : \mu \le_{cx} \nu \}$.
  \end{corollary}

\subsection{Shadow couplings} \label{ssec:main.SC}
  The shadow couplings introduced by Beiglb\"ock and Juillet in \cite{BeJu21} fit into this framework. These couplings admit characterizations in terms of optimality, in terms of geometry of their support sets, and as barrier-type solutions to the Skorokhod embedding problem (c.f.\ \cite[Theorem 1.1]{BeJu21}). Let ${\cal U}=[0,1]$ and $\Leb$ denote the Lebesgue measure on this set. For $\mu,\nu\in{\cal P}_1(\R)$ such that $\mu\le_{cx}\nu$, shadow couplings are solutions to weak martingale transport problems of the form
\begin{equation} \label{eq:SC.WMOT}
    \inf_{\pi \in \Pi_M(\mu,\nu)} \int C_{\bar \mu}(x,\pi_x) \, \mu(dx),
\end{equation}
where $\bar\mu \in \Pi(\mu,\Leb)$ and $C_{\bar \mu}(x,\rho) := \inf_{\chi \in \Pi_M(\delta_x \times \bar \mu_x, \rho)} \int (1 - u) \sqrt{1 + y^2} \, \chi(dx,du,dy)$.
The unique solution $\tilde \pi^\star$ to \eqref{eq:SC.WMOT} is called a shadow coupling with source $\bar \mu$.
Attached to each shadow coupling $\tilde \pi^\star$ with source $\bar \mu$ is the (unique) lifted shadow coupling $\pi^\star \in \Pi_M(\bar\mu,\nu)$ that satisfies $\proj_{1,3} \pi^\star = \tilde \pi^\star$ and minimizes
\begin{equation}
    \label{eq:SC.liftedMOT}
    V_{\rm SC}(\bar\mu,\nu) := \inf_{\gamma \in \Pi_M(\bar\mu,\nu)} \int_{\R \times [0,1] \times \R} (1 - u) \sqrt{1 + y^2} \, \gamma(dx,du,dy).
\end{equation}
We denote by $\SC$ the selector of unique optimizers of \eqref{eq:SC.liftedMOT}, i.e., $\SC(\bar\mu,\nu) = \pi^\star$.
According to \cite[Theorem 1.1]{BeJu21}, there is a stochastic basis supporting an $(\mathcal F_t)_{t \in [0,1]}$-Brownian motion $B = (B_t)_{t \in [0,1]}$ and an $\mathcal F_0$-measurable random variable $U$ uniformly distributed on $[0,1]$ so that $\pi^\star \sim (B_0,U,B_\tau)$ where $\tau$ is the hitting time of the process $(B_t,U)_{t \in [0,1]}$ into a left barrier, that is a Borel set $R \subseteq [0,1] \times \R$ such that $(x,u) \in R$, $v \le u$ implies $(x,v) \in R$.
Put differently, there exist two Borel measurable maps $T_1, T_2 \colon \R \times [0,1] \to \R$ satisfying
\[
    \forall x\in\R,\;\forall 0 \le v \le u \le 1,\;T_1(x,u) \le T_1(x,v) \le x \le T_2(x,v) \le T_2(x,u),
\]
such that
\begin{equation}
    \label{eq:SC.extremal}
    \pi^\star(dx,du,dy) = \bar \mu(dx,du) B(x,T_1(x,u),T_2(x,u))(dy),
\end{equation}
where, for $x,l,r \in \R$, $B(x,l,r)$ denotes the distribution
\[
    B(x,l,r) := 
    \begin{cases}
        \frac{r-x}{r-l}\delta_{l}+\frac{x-l}{r-l}\delta_{r} & l < x < r, \\
        \delta_x & \text{else.}
    \end{cases}
\]
W.l.o.g. we suppose that $T_1(x,u)=x=T_2(x,u)$ as soon as either $T_1(x,u)=x$ or $T_2(x,u)=x$.
 \begin{proposition}
     \label{prop:shadcoupl}
     The optimal value map $V_{\rm SC}$, the selector $\SC$ of optimizers of \eqref{eq:SC.liftedMOT}, and $J \circ \SC$ are continuous on the domain $\{ (\bar\mu,\nu) : \proj_1 \bar \mu \le_{cx} \nu, \proj_2 \bar \mu = \Leb\} \subseteq \mathcal P_p(\R \times [0,1]) \times \mathcal P_p(\R)$ and with range $\R$, $\mathcal P_p(\R \times [0,1] \times \R)$, and $\mathcal P_p(\R \times [0,1] \times \mathcal P_p(\R))$, respectively.
     
    Furthermore, when $(\bar \mu^k,\nu^k)_{k \in \N}$ with $\proj_1 \bar \mu^k \le_{cx} \nu^k$ and $\proj_2 \bar\mu^k = \Leb$ is a sequence converging to $(\bar\mu,\nu)$ in ${\cal P}_p(\R\times\U)\times{\cal P}_p(\R)$ with $\bar\mu^k \to \bar\mu$ in total variation and $(T^k_1,T^k_2)$ (resp. $(T_1,T_2)$) are the pairs of maps satisfying \eqref{eq:SC.extremal} for $\SC(\bar\mu^k,\nu^k)$ and (resp. $\SC(\bar\mu,\nu)$) respectively, then
    \begin{align*}
        (T_1^k,T_2^k) \to (T_1,T_2)\quad \text{in }\bar\mu\text{-probability on }\{T_1 \neq T_2\}, \\
        |T_1^k - T_1| \wedge |T_2^k - T_2| \to 0 \quad\text{in }\bar\mu\text{-probability on }\{T_1 = T_2 \}.      
    \end{align*}
 \end{proposition}

\section{Proofs} \label{sec:proofs}
\subsection{Topological refinements} \label{ssec:topology}
    
In order to prove Proposition \ref{prop:stability.VIX}, we introduce refinements of the weak topology as detailed below, which we use to establish stronger versions of the results given in the introduction.
For the rest of the paper, let $\X$ and $\Y$ be (non-empty) Polish subsets of $\R$ and consider two growth functions ${\bar f} \colon \X \times \U \to [1,+\infty)$ and $g \colon \Y \to [1,+\infty)$ that are both continuous and
\begin{equation}
    \label{eq:growth.function}
    \liminf_{ \substack{ |x| \to \infty \\ x \in \X} } \inf_{u \in \U} \frac{{\bar f}(x,u)}{|x|} > 0
    \quad\text{and}\quad
    \liminf_{ \substack{ |y| \to \infty \\ y \in \Y} } \frac{g(y)}{|y|} > 0.
  \end{equation}
We define the sets $\mathcal P_{\bar f}(\X \times \U) := \{ \rho \in \mathcal P(\X\times \U) : \rho({\bar f}) < \infty \}$ and $\mathcal P_g(\Y) := \{ \rho \in \mathcal P(\Y) : \rho(g) < \infty\}$ and endow them with the initial topology induced by $C_b(\X \times \U) \cup \{ {\bar f} \}$ resp.\ $C_b(\Y) \cup \{ g \}$, that is,
\begin{align*}
    \rho^k \to \rho \text{ in } \mathcal P_{\bar f}(\X \times \U) \iff& \rho^k \to \rho \text{ weakly and }\rho^k({\bar f}) \to \rho({\bar f}), \\
    \rho^k \to \rho \text{ in } \mathcal P_g(\Y) \iff& \rho^k \to \rho \text{ weakly and }\rho^k(g) \to \rho(g).    
\end{align*}
Similarly, we define $\mathcal P_{{\bar f} \oplus g}(\X \times \U \times \Y) = \{ \rho \in \mathcal P(\X \times \U \times \Y) : \rho({\bar f} \oplus g) < \infty \}$ and $\mathcal P_{{\bar f} \oplus \hat g}(\X \times \U \times \mathcal P_g(\Y)) := \{ \rho \in \mathcal P(\X \times \U \times \mathcal P_g(\Y)) : \rho({\bar f} \oplus \hat g) < \infty\}$ where ${\bar f} \oplus g(x,u,y) := {\bar f}(x,u) + g(y)$ and ${\bar f} \oplus \hat g(x,u,\rho) = {\bar f}(x,u) + \rho(g)$.
Again, these spaces are endowed with the topology induced by $C_b(\X \times \U \times \Y) \cup \{ {\bar f} \oplus g \}$ resp.\ $C_b(\X \times \U \times \mathcal P_g(\Y)) \cup \{ {\bar f} \oplus \hat g \}$, that is,
\begin{align*}
    \rho^k \to \rho \text{ in } \mathcal P_{{\bar f} \oplus g}(\X \times \U \times \Y) \iff& \rho^k \to \rho \text{ weakly and }\rho^k({\bar f} \oplus g) \to \rho({\bar f} \oplus g), \\
    \rho^k \to \rho \text{ in } \mathcal P_{{\bar f} \oplus \hat g}(\X \times \U \times \mathcal P_g(\Y)) \iff& \rho^k \to \rho \text{ weakly and }\rho^k({\bar f} \oplus\hat g) \to \rho({\bar f} \oplus \hat g).    
\end{align*}
Note that when $\X=\R=\Y$ and ${\bar f}(x,u) = 1 + |x|^p + d_\U^p(u_0,u)$ for some $u_0 \in \U$ and $g(y) = 1 + |y|^p$, we have $\mathcal P_{\bar f}(\X \times \U) = \mathcal P_p(\X \times \Y)$, $\mathcal P_g(\Y) = \mathcal P_p(\Y)$, and the topologies on the above introduced spaces coincide with the corresponding $p$-Wasserstein topologies.
Moreover, when $d_\U$ is bounded, the growth condition \eqref{eq:growth.function} provides that these topologies are finer than the corresponding $1$-Wasserstein topology.
The reader may ignore these refinements of the wea k topology and may accordingly substitute in every statement these refinements with a $p$-Wasserstein topology.

Next, we define the injection
\begin{align}
    \label{eq:J}
    \begin{split}
    J \colon \mathcal P_{{\bar f} \otimes g}(\X \times \U \times \Y) &\to \mathcal P_{{\bar f} \otimes \hat g} (\X \times \U \times \mathcal P_g(\Y)),
    \\ 
    \pi &\mapsto (x,u,\pi_{x,u})_\# \pi,
    \end{split}
\end{align}
and observe that $\int C(x,u,\pi_{x,u}) \, \proj_{\X \times \U} \pi(dx,du) = J(\pi)(C)$ for any $J(\pi)$-integrable $C \colon \X \times \U \times \mathcal P(\Y) \to \R \cup \{ \infty \}$.
In our specific setting we treat the $\X$- and $\U$-coordinates similarly as we interpret the $\X$-coordinate as the spatial state (at time 1) and the $\U$-coordinate as the information state (at time 1), whereas we think of the $\Y$-coordinate as the state at time 2.
For this reason, we say a sequence $(\pi^k)_{k \in \N}$ in $\mathcal P(\X \times \U \times \Y)$ converges in the adapted weak topology to $\pi$ if
\begin{equation}
    \label{eq:adapted.weak.convergence}
    J(\pi^k) \to J(\pi) \quad \text{in }\mathcal P(\X \times \U \times \mathcal P(\Y)).
\end{equation}
The associated adapted $p$-Wasserstein distance of $\pi^1$ and $\pi^2$, where $\pi^1,\pi^2 \in \mathcal P_p(\X \times \U \times \Y)$, is given by
\begin{equation}
    \label{eq:AW}
    \AW_p^p(\pi^1,\pi^2) := \inf_{ \chi \in \Pi(\proj_{\X\times\U} \pi^1, \proj_{\X\times\U} \pi^2)} \int d_\X^p(x_1,x_2) + d_\U^p(u_1,u_2) + \W_p^p(\pi^1_{x_1,u_1}, \pi^2_{x_2,u_2}) \, \chi(dx_1,du_1,dx_2,du_2)
\end{equation}
and satisfies $\AW_p^p(\pi^1,\pi^2) = \W_p^p(J(\pi^1),J(\pi^2))$ where $\W_p$ is the $p$-Wasserstein distance on $\mathcal P_p(\X \times \U \times \mathcal P_p(\Y))$.

The following reformulation of \cite[Lemma 2.7]{Ed19} proves very useful to check convergence in the adapted Wasserstein topology.
\begin{lemma}\label{lem:eder}
   Let $({\cal V},d_{\cal V})$ and $({\cal Z},d_{\cal Z})$ be Polish metric spaces, $\mu \in{\cal P}_p(\cal V)$ and  $\varphi\colon\cal V\to\cal Z$ by a measurable function such that $\varphi_\#\mu\in{\cal P}_p(\cal Z)$. 
   Then $\Pi(\mu,\mu)\ni \pi\mapsto \int_{\cal V\times\cal V}d^p_{\cal Z}(\varphi(v),\varphi(\hat v))\,\pi(dv,d\hat v)$ vanishes with $\int_{\cal V\times\cal V}d_{\cal V}^p(v,\hat v) \, \pi(dv,d\hat v)$.
\end{lemma}
For more details on the adapted weak topologies and the adapted Wasserstein distance, we refer to \cite{BaBaBeEd19a, BaBePa21}.

\subsection{Convergence of subprobability measures}\label{ssec:subprobability}
Occasionally it will be advantageous to work with subprobability measures.
Therefore, we denote by $\mathcal M_p(\X)$ the set of finite non-negative Borel measures on $\X$ that have finite $p$-th moments and by $\mathcal M_p^\ast(\X)$ the subset of measures with positive mass.
We say that  a sequence $(\rho^k)_{k \in \N}$ converges in ${\cal M}_p(\X)$ to $\rho$ if one of the following equivalent conditions holds:
\begin{enumerate}[label = (\alph*)]
    \item $(\rho^k)_{k \in \N}$ converges weakly to $\rho$ and, for some $x_0 \in \X$, $\lim_{k \to \infty} \int d_\X^p(x,x_0) \, \rho^k(dx) = \int d_\X^p(x,x_0) \, \rho(dx)$;
    \item for every continuous function $\varphi \colon \X \to \R$ such that, for some $x_0 \in \X$ and all $x \in \X$, $|\varphi(x)| \le 1 + d_\X^p(x,x_0)$, $\lim_{k \to \infty} \rho^k(\varphi) = \rho(\varphi)$.
\end{enumerate}
Further, when $\rho,\tilde \rho \in \mathcal M_p^\ast(\X)$ have equal mass, we can consider their $p$-Wasserstein distance given by
\[
    \mathcal W_p(\rho,\tilde \rho) := \rho(\X)^\frac1p \, \mathcal W_p\left( \frac{\rho}{\rho(\X)}, \frac{\tilde \rho}{\tilde \rho(\X)}\right),
\]
and similarly define the $p$-adapted Wasserstein distance $\AW_p$ between measures $\pi, \tilde \pi \in \mathcal M_p^\ast(\X \times U \times \Y)$ that have equal mass.

\begin{lemma}
	\label{lem:convergence_subprobabilities}
	Let $p \ge 1$, $\rho \in \mathcal M_p^\ast(\X)$, and $(\rho^k)_{k \in \N}$ be a sequence in $\mathcal M_p^\ast(\X)$ with $\lim_{k \to \infty} \rho^k(\X) = \rho(\X)$.
	Then the following are equivalent:
	\begin{enumerate}[label = (\roman*)]
		\item \label{it:convergence_sub_1}
        $(\rho^k)_{k \in \N}$ converges in $\mathcal M_p(\X)$ to $\rho$;
		\item \label{it:convergence_sub_2}
		the normalized sequence $\left(\frac{\rho^k}{\rho^k(\X)}\right)_{k \in \N}$ converges to $\frac{\rho}{\rho(\X)}$ in $\mathcal P_p(\X)$.
	\end{enumerate}
\end{lemma}

\begin{proof}
    Since $\lim_{k \to \infty} \rho^k(\X) = \rho(\X)$, we have in either case that $(\rho^k)_{k \in \N}$ and the normalized sequence $(\rho^k / \rho^k(\X))_{k \in \N}$ are weakly convergent with limit $\rho$ and $\rho / \rho(\X)$, respectively.
	For some $x_0 \in \X$, we then have
	\[
		\lim_{k \to \infty} \int d_\X(x,x_0)^p \, \rho^k(dx) =\int d_\X(x,x_0)^p \, \rho(dx) 
		\iff
		\lim_{k \to \infty} \int d_\X(x,x_0)^p \, \frac{\rho^k(dx)}{\rho^k(\X)} =\int d_\X(x,x_0)^p \, \frac{\rho(dx)}{\rho(\X)}.
	\]	
	Thus, the equivalence of \ref{it:convergence_sub_1} and \ref{it:convergence_sub_2} follows from \cite[Definition 6.8]{Vi09}.
\end{proof} 

\begin{lemma}
	\label{lem:convergence_dominated}
	Let $p \ge 1$ and $\X$ be a Polish space.
	Let $(\rho ^k)_{k \in \N}$ be a convergent sequence in $\mathcal M_p(\X)$ and
	$(q^k)_{k \in \N}$ be a weakly convergent sequence with $q^k \le \rho^k$ for every $k\in\N$.
	Then, $(q^k)_{k \in \N}$ converges in $\mathcal M_p(\X)$.
\end{lemma}

\begin{proof}
	Write $\rho$ and $q$ for the weak limits of $(\rho^k)_{k \in \N}$ and $(q^k)_{k \in \N}$ respectively.
	Consider the sequence $\tilde q^k := \rho^k - q^k \in \mathcal M_p(\X)$, $k \in \N$, which is also weakly convergent with limit $\tilde q := \rho - q$.
	By Portmanteau's theorem we have
	\begin{align*}
		\int d_\X(x,x_0)^p \, q(dx) &\le \liminf_{k \to \infty} \int d_\X(x,x_0)^p \, q^k(dx), \\
		\\
		\int d_\X(x,x_0)^p \,\tilde q(dx) &\le \liminf_{k \to \infty} \int d_\X(x,x_0)^p \, \tilde q^k(dx).
	\end{align*}
	Hence,
	\begin{align*}
		\limsup_{k \to \infty} \int d_\X(x,x_0)^p \, q^k(dx) 
		&=
		\int d_\X(x,x_0)^p \, \rho(dx) - \liminf_{k \to \infty}\int d_\X(x,x_0)^p \, \tilde q^k(dx) 
		\\
		&\le
		\int d_\X(x,x_0)^p \, \rho(dx) - \int d_\X(x,x_0)^p \, \tilde q(dx) = \int d_\X(x,x_0)^p \, q(dx),
	\end{align*}
	yields $\lim_{k \to \infty} \int d_\X(x,x_0)^p \, q^k(dx) = \int d_\X(x,x_0)^p \, q(dx)$.
\end{proof}

\subsection{Approximation of extended martingale couplings: proof of Theorem \ref{thm:intro.approximation}} \label{ssec:proof.approximation}
Before stating and proving a strengthened  version of Theorem  \ref{thm:intro.approximation}, let us deduce stability of the set of martingale couplings with respect to the marginals.
The Hausdorff distance between two closed subsets $\mathcal A, \mathcal B \subseteq \mathcal P_p(\R \times \U \times \R)$ is denoted by $d_{\rm H}(\mathcal A, \mathcal B) :=\max \left( \sup_{a \in \mathcal A} \mathcal W_p(a, \mathcal B), \sup_{b \in \mathcal B} \mathcal W_p(b,\mathcal A) \right)$ where $\mathcal W_p(a,\mathcal B) := \inf_{b \in \mathcal B} \mathcal W_p(a,b)$. Note that for $(\bar\mu,\nu)\in{\cal P}_p(\R\times\U)\times {\cal P}_p(\R)$, $\Pi_M(\bar\mu,\nu)$ is a compact subset of ${\cal P}_p(\R\times\U\times\R)$.
\begin{corollary}\label{cor:haus}
   Let $((\bar\mu^k,\nu^k))_{k\in\N}$ with $\proj_1\bar\mu^k\le_{cx}\nu^k$  be convergent in $\mathcal P_p(\R\times\U)\times\mathcal P_p(\R)$ to $(\bar\mu,\nu)$. Then \begin{equation*}
    \lim_{k \to \infty} d_{\rm H}\left( \Pi_M(\bar \mu^k, \nu^k), \Pi_M(\bar \mu,\nu) \right) = 0.
\end{equation*}
\end{corollary}
The corresponding statement for couplings without the martingale constraint is straightforward to see as in this case one even has
\[
    d_{\rm H}\left( \Pi(\bar\mu,\nu), \Pi(\bar\mu',\nu') \right) \le \left(\mathcal W^p_p(\bar\mu,\bar\mu') + \mathcal W^p_p(\nu,\nu')\right)^{1/p}\le \mathcal W_p(\bar\mu,\bar\mu') + \mathcal W_p(\nu,\nu').
\]
\begin{proof}[Proof of Corollary \ref{cor:haus}]
Assume that $(\bar \mu^k,\nu^k)_{k \in \N}$ converge in $\mathcal P_p(\R \times \U) \times \mathcal P_p(\R)$ to $(\bar \mu,\nu)$.
Note that $\bigcup_{k \in \N} \Pi_M(\bar \mu^k,\nu^k)$ is relatively compact as consequence of Prokhorov's theorem. 

On the one hand, any sequence $(\pi^k)_{k \in \N}$ with $\pi^k \in \Pi_M(\bar \mu^k,\nu^k)$ admits a weakly convergent subsequence $(\pi^{k_j})_{j \in \N}$ with limit $\pi \in \Pi_M(\bar \mu,\nu)$. Therefore,
\[
    \lim_{k \to \infty} \sup_{\pi^k \in \Pi(\bar \mu^k,\nu^k)} \mathcal W_p\left(\pi^k,\Pi_M(\bar \mu,\nu)\right) = 0.
\]

On the other hand, the map $\pi \mapsto \mathcal W_p(\pi, \Pi_M(\bar \mu^k,\nu^k))$ is $\mathcal W_p$-continuous.
Thus, by compactness of the set of martingale couplings 
there is for every $k \in \N$, $\tilde \pi^k \in \Pi_M(\bar \mu,\nu)$ with $\mathcal W_p\left(\tilde \pi^k, \Pi_M(\bar \mu^k,\nu^k) \right) = \sup_{\pi \in \Pi_M(\bar \mu,\nu)} \mathcal W_p\left(\pi, \Pi_M(\bar \mu^k,\nu^k) \right)$.
Again by compactness, any subsequence of $(\tilde \pi^k)_{k \in \N}$ admits a further subsequence converging weakly to some limit in $\Pi_M(\bar \mu,\nu)$.
For any of these accumulation points there is an approximative sequence provided by Theorem \ref{thm:intro.approximation}.
Consequently,
\[
    \lim_{k \to \infty} \sup_{\pi \in \Pi_M(\bar \mu,\nu) } \mathcal W_p\left( \pi, \Pi_M(\bar \mu^k,\nu^k) \right) = 0.\qedhere
\]
\end{proof}
We will prove the following strengthened version of Theorem \ref{thm:intro.approximation} which takes into account general integrability conditions over Polish subsets of $\R$ and is, in fact, an extension of the main result in \cite{BJMP22}. For $\mu\in{\cal P}(\X)$ and $\nu\in{\cal P}_g(\Y)$, $\mu\le_{cx}\nu$ means that the respective extensions $\mu(\cdot\cap \X)$ and $\nu(\cdot\cap\Y)$ of $\mu$ and $\nu$ to the Borel sigma-field on $\R$ satisfy $\mu(\cdot\cap \X)\le_{cx}\nu(\cdot\cap\Y)$.

\begin{theorem} \label{thm:approximation}
    Let $(\bar \mu^k,\nu^k)_{k \in \N}$, $\proj_1 \bar \mu^k \le_{cx} \nu^k$, be a convergent sequence in $\mathcal P_{\bar f}(\X \times \U) \times \mathcal P_g(\Y)$ with limit $(\bar \mu,\nu)$.
    Then, every coupling $\pi \in \Pi_M(\bar \mu,\nu)$ is the limit in the adapted weak topology of a sequence $(\pi^k)_{k \in \N}$ with $\pi^k \in \Pi_M(\bar \mu^k,\nu^k)$.
\end{theorem}

The proof of Theorem \ref{thm:approximation} relies on the next three auxiliary results, that are Lemma \ref{lem:core_theorem_single_irreducible_to_general}, Lemma \ref{lem:core_theorem_finite_sum}, and Proposition \ref{prop:approximation_finite_pairs}.

Here we recall the notion of a pair of measures being \emph{irreducible} and refer to \cite[Appendix A]{BeJu16} for further details.
When $\mu \in \mathcal M_1(\R)$, we denote by $u_\mu$ its potential function, that is the map defined by $u_\mu(y) = \int_\R |y-x| \, \mu(dx)$ for $y \in \R$.
A pair $(\mu,\nu)$ of finite positive measures in convex order is called irreducible if $I := \{ x \in \R : u_\mu(x) < u_\nu(x) \}$ is an interval and $\mu(I) = \mu(\R)$.
For any pair $(\mu,\nu) \in \mathcal P_1(\R)^2$, $\mu \le_{cx} \nu$, there exists $N \subseteq \mathbb N$ and a sequence $(\mu_n,\nu_n)_{n \in N}$ of irreducible pairs of subprobability measures in convex order such that
\begin{equation} \label{eq:}
    \mu = \eta + \sum_{n \in N} \mu_n\quad \text{and}\quad\nu = \eta + \sum_{n \in N} \nu_n,
\end{equation}
where $(\{ u_{\mu_n} < u_{\nu_n} \})_{n \in N}$ are pairwise disjoint and $\eta = \mu |_{\{ u_\mu = u_\nu\}}$, c.f.\ \cite[Theorem A.4]{BeJu16}.
The sequence $(\mu_n,\nu_n)_{n \in N}$ is unique up to rearrangement of the pairs and is called the decomposition of $(\mu,\nu)$ into irreducible components.
The next lemma generalizes \cite[Proposition 2.4]{BJMP22} to the current setting where we have to keep track of the extra coordinate $u \in \U$.

\begin{lemma}\label{lem:core_theorem_single_irreducible_to_general} 
Let $(\bar \mu^k,\nu^k)_{k \in \N}$, $\proj_1 \bar \mu^k \le_{cx} \nu^k$, be a convergent sequence in $\mathcal P_1(\R \times \U) \times \mathcal P_1(\R)$ with limit $(\bar \mu,\nu)$ and write $\proj_1 \bar \mu =: \mu$.
Let $(\mu_n, \nu_n)_{n \in N}$ where $N \subseteq \N$ be the decomposition of $(\mu,\nu)$ into irreducible components.
Then, for every $k \in \N$, there exists a decomposition of $(\bar\mu^k,\nu^k)$ into pairs of subprobability measures $(\bar\mu^k_n,\nu^k_n)_{n\in N}$, $(\bar\eta^k,\upsilon^k)$ such that $\proj_1 \bar \mu^k_n\le_c\nu^k_n$ for each $n\in N$, $\proj_1 \bar \eta^k\le_c\upsilon^k$ and 
\begin{gather*}
   \bar\eta^k+\sum_{n\in N}\bar\mu^k_n=\bar\mu^k,\quad \upsilon^k+\sum_{n\in N}\nu^k_n=\nu^k,\\ 
   \lim_{k \to \infty} \sum_{n \in N} \W_1(\bar\mu^k_n,\mu_n \times\bar\mu_x)+\W_1(\nu^k_n,\nu_n) + \W_1(\bar\eta^k,\eta\times\bar\mu_x)+\W_1(\upsilon^k,\eta) = 0.
\end{gather*}
\end{lemma}

\begin{proof}
We set $\bar\eta(dx,du)=\eta(dx)\bar\mu_x(du)$ and $\bar\mu_n(dx,du)=\mu_n(dx)\bar\mu_x(du)$ for $n\in N$. Let us select $\W_1$-optimal couplings $\hat \pi^k \in \Pi(\bar \mu,\bar \mu^k )$ and define
	\[
		\bar \mu_n^k(dx,du) := \int_{\hat x, \hat u} \hat \pi^k_{\hat x,\hat u}(dx,du)	\, \bar \mu_n(d\hat x,d\hat u), \quad
		\bar \eta^k(dx,du) := \int_{\hat x, \hat u} \hat \pi^k_{\hat x, \hat u}(dx,du) \, \bar \eta(d\hat x,d\hat u).
	\]
We have $$\W_1(\bar \eta^k,\bar \eta)+\sum_{n\in N}\W_1(\bar \mu^k_n,\bar \mu_n)\le\W_1(\bar \mu^k,\bar \mu)\underset{k\to+\infty}{\longrightarrow}0.$$
	Pick any $\bar \pi^k \in \bar\Pi_M(\bar \mu^k, \nu^k)$, set
	\begin{gather*}
		\nu_n^k(dy) := \int_{x,u} \bar \pi^k_{x,u}(dy) \, \bar \mu^k_n(dx,du), \quad
		\upsilon^k(dy) := \int_{x,u} \bar \pi^k_{x,u}(dy) \, \bar \eta^k(dx,du),
		\\
		\bar \pi_n^k(dx,du,dy) := \bar \mu^k_n(dx,du) \bar \pi^k_{x,u}(dy) \in \Pi_M(\bar \mu^k_n, \nu^k_n).
	\end{gather*}
        Let $n\in N$. Since $(\nu^k = \upsilon^k + \sum_{n \in N} \nu_n^k)_{k \in \N}$ converges to $\nu$ in $\mathcal W_1$, we have tightness of $(\nu^k_n)_{k \in \N}$.
		As the marginals are tight, we can pass by Prokhorov's theorem to a subsequence and assume that, as $k\to\infty$,  $(\bar \pi^k_n)_{k \in \N}$ converges weakly to $\tilde \pi_n \in \Pi_M(\bar \mu_n, \tilde \nu_n)$ where $\tilde \nu_n \in \mathcal M_1(\R)$.
		At the same time, as $(\bar \mu^k)_{k \in \N}$ and $(\nu^k)_{k \in \N}$ are $\mathcal W_1$-convergent, by passing to a subsequence we can additionally assume convergence of $(\bar \pi^k)_{k \in \N}$ in $\mathcal W_1$ to some $\bar \pi \in \Pi_M(\bar \mu,\nu)$.
		Since $\bar\pi^k_n\le \bar\pi^k$ for each $k\in\N$, passing to weak limits, we have $\tilde \pi_n\le \bar\pi$, which, in view of $\tilde \pi_n(dx,du,\R)=\bar \mu_n(dx,du)$, implies $\tilde \pi_n(dx,du,dy) = \bar \mu_n(dx,du) \bar \pi_{x,u}(dy)$ so that $\tilde \nu_n=\nu_n$. By Lemma \ref{lem:convergence_dominated}, we get that $\mathcal W_1(\bar \pi^k_n,\tilde \pi_n)$ goes to $0$ as $k\to\infty$. Hence, $\mathcal W_1(\nu^k_n,\nu_n)$ also goes to $0$.		
		The same argument applies to deal with $\mathcal W_1(\upsilon^k,\eta)$.
\end{proof}

In order to show Theorem \ref{thm:approximation}, it turns out to be beneficial to first demonstrate that a family of couplings with a simpler structure is already dense.
We say that a coupling $\pi \in \Pi_M(\bar \mu,\nu)$ is simple if there is $J \in \N$, a measurable partition $(\U_j)_{j = 1}^J$ of $\U$ into $\proj_2 \bar \mu$-continuity sets and, for $j \in \{1,\ldots,J\}$, a martingale kernel $(K_{j,x})_{x \in \R}$ such that
\begin{equation}\label{eq:bar_pi_finite_sum}
		\pi(dx,du,dy) = \sum_{j = 1}^J \mathbbm 1_{\U_j}(u)\bar \mu(dx,du) \, K_{j,x}(dy).
\end{equation}
Put differently, one may say $\pi$ is simple if there exist (classical) martingale couplings $\pi^j \in \Pi_M(\mu,\nu_j)$, $j \in \{1,\ldots,J\}$, and a measurable partition $(\U_j)_{j = 1}^J$ of $\U$ in $\proj_2 \bar \mu$-continuity sets such that 
\[ \pi(dx,du,dy) = \sum_{j = 1}^J \pi^j(dx,dy) \bar \mu_x(du \cap \U_j). \]
The next lemma establishes that these simple couplings are already dense in $\Pi_M(\bar\mu,\nu)$.

\begin{lemma}\label{lem:core_theorem_finite_sum}
Let $\bar\mu\in\mathcal P_1(\R\times\U)$ and $\nu\in\mathcal P_1(\R)$.
Then, the set of couplings satisfying \eqref{eq:bar_pi_finite_sum} is dense in $\Pi_M(\bar \mu,\nu)$ w.r.t.\ the adapted weak topology.
\end{lemma}

\begin{proof}
  We denote by $\lambda=\proj_2\bar\mu \in \mathcal P_1(\U)$.
  Let $u_0\in \U$ and $\varepsilon>0$.
  We claim that there is a finite partition $(\U_j)_{j = 1}^J$, $J \in \N$, of $\U$ into $\lambda$-continuity sets such that
  \begin{equation}
	\label{eq:core_thm_fin_sum.prop}
	\sup \{ d_\U(u, \hat u) \colon u, \hat u \in \U_j \} \le \frac{\epsilon}{2}
	\mbox{ for }
	j \in \{1,\ldots,J - 1\},
	\mbox{ and }
	\int_{\U_J} d_\U(u,u_0) \, \lambda(du) \le \frac{\epsilon}{4}.
  \end{equation}
  To this end, note that since the map $u \mapsto d_\U(u,u_0)$ is integrable w.r.t.\ $\lambda$, we can choose $M_\epsilon \in [0,+\infty)$ with
  \[
	\int_{\{u \in \U \colon d_\U(u,u_0)>M_\varepsilon\}}d_\U(u,u_0) \, \lambda(du)\le \frac{\varepsilon}{8}.
  \]
  By inner regularity of $\lambda$ there exists a compact subset $\mathcal K$ of $B := \{ u \in \U \colon d_\U(u,u_0) \le M_\epsilon \}$ such that $\lambda(B \setminus \mathcal K)\le\frac{\varepsilon}{8M_\varepsilon}$.
  Therefore, we have
  \[
  	\int_{B \setminus \mathcal K} d_\U(u,u_0) \, \lambda(du) \le \lambda(B \setminus \mathcal K) M_\epsilon \le \frac{\epsilon}{8}.
  \]
  Next, we choose for each $u \in \mathcal K$ a radius $r_u \in (0,\frac{\epsilon}{4}]$ such that the boundary of the ball $B_{r_u}(u) := \{ \hat u \in \U \colon d_\U(u,\hat u) < r_u \}$ has zero measure under $\lambda$.
  The family $(B_{r_u}(u))_{u \in \mathcal K}$ is an open cover of the compact set $\mathcal K$, which permits us to extract from this family a finite subcover of $\mathcal K$ denoted by $(B_j)_{j = 1}^I$, $I\in \N$.
  Let $J := I + 1$, $\U_J := \bigcap_{j = 1}^J B_j^\textrm{c}\subset \mathcal K^c$, and set recursively, $\U_j := B_j \cap\left\{\bigcup_{i = 1}^{j - 1} B_i\right\}^c$ for $j \in \{1,\ldots,J-1\}$.
  By this procedure we have constructed a partition $(\U_j)_{j = 1}^J$ of $\U$ into measurable sets.
  Moreover, as for each $i \in \{1,\ldots,J\}$ the boundary of $\U_i$ is contained in the union of the boundaries of the balls $(B_j)_{j = 1}^J$, it must have zero $\lambda$-measure.
  Finally, for each $j \in \{1,\ldots,J-1\}$ we get
  \[
  	\sup \{ d_\U(u, \hat u) \colon u, \hat u \in \U_j \} \le \sup \{ d_\U(u,\hat u) \colon u, \hat u \in B_j \} \le \frac{\epsilon}{2},
  \]
  and compute
  \[
  	\int_{\U_J} d_\U(u,u_0) \, \lambda(du) \le \int_{B^\textrm{c}} d_\U(u,u_0) \, \lambda(du) + \int_{B \setminus \mathcal K} d_\U(u,u_0) \, \lambda(du) \le \frac{\epsilon}{8} + \frac{\epsilon}{8} = \frac{\epsilon}{4}.
  \]
  We have shown the claim \eqref{eq:core_thm_fin_sum.prop}.

  Consider the disintegration kernel $\bar \mu_x$ such that $\bar \mu(dx,du) = \mu(dx) \bar \mu_x(du)$, and, for $j \in \{1,\ldots,J\}$, denote by $\bar \mu_j(dx,du)$ the restrictions $\mathbbm 1_{\R\times \U_j }(x,u) \bar \mu(dx,du)$.
  Since $\bar \mu_j(dx,du)$-a.e.\ $\bar \mu_x(\U_j) > 0$ we can define
  \[
  	\bar \pi^J(dx,du,dy) := \sum_{j = 1}^J \bar \mu_j(dx,du) K_{j,x}(dy)
	\mbox{ where }
	K_{j,x}(dy) := \int_{\U_j} \bar \pi_{x,u}(dy) \, \frac{\bar \mu_x(du)}{\bar \mu_x(\U_j)},
  \]
  and remark that $\bar \pi^J \in \Pi_M(\bar \mu,\nu)$.
  The last step is to estimate the $\AW_1$-distance between $\bar \pi$ and $\bar \pi^J$.

  Using that $(\id,\id)_\# \bar \mu$ is an admissible coupling in the minimization problem that constitutes the $\AW_1$-distance between $\bar \pi$ and $\bar \pi^J$, and Jensen's inequality (see for example \cite[Proposition A.9]{BeJoMaPa21b}), we obtain the estimate
  \begin{align*}
	\nonumber
	\AW_1(\bar \pi, \bar \pi^J) &\le \sum_{j = 1}^J \int_{\R\times\U_j} \W_1(\bar \pi_{x,u}, K_{j,x}) \, \bar\mu(dx,du)
	\\
	&\le
\sum_{j = 1}^J \int_{\R\times\U_j}  \int_{\U_j} \W_1(\bar \pi_{x,u}, \bar \pi_{x,\hat u}) \, \frac{\bar \mu_{x}(d \hat u)}{\bar \mu_x(\U_j)} \, \bar \mu(dx,du)\\
    &=\int \mathcal W_1(\bar \pi_{x,u}, \bar \pi_{\hat u, \hat x}) \, \gamma(dx,du,d \hat x,d\hat u) 
  \end{align*}
  where $\gamma := \sum_{j = 1}^J \gamma_j$ with the subprobability couplings $\gamma_j$ defined for  $j \in \{1,\ldots, J\}$ by
\[
	\gamma_j := ((x,u,\hat u) \mapsto (x,u,x,\hat u))_\# \left( \frac{\bar \mu_x(d \hat u)}{\bar \mu_x(\U_j)} \bar \mu_j(dx,du) \right) \in \Pi(\bar \mu_j, \bar \mu_j).
\]
We have $\gamma \in \Pi(\bar \mu, \bar \mu)$ and, by \eqref{eq:core_thm_fin_sum.prop},
\begin{align}
	\nonumber
	\int d_\U(u,\hat u) + |x - \hat x| \, &\gamma(dx,du,d\hat x,d\hat u)
	=
	\sum_{j = 1}^J \int d_\U(u,\hat u) \, \gamma_j(dx,du,d\hat x,d\hat u)
	\\
	\nonumber
	&\le \frac{\epsilon}{2} \sum_{j = 1}^{J-1} \gamma_j(\R \times \U \times \R\times \U )+\int d_\U(u, u_0) + d_\U(u_0, \hat u) \, \gamma_J(dx,du,d\hat x,d\hat u) \\	
	\label{eq:core_thm_fin_sum.est3}
	&\le
\frac{\epsilon}{2}+	\frac{\epsilon}{4} + \frac{\epsilon}{4}   = \epsilon.
\end{align}
Since $\epsilon>0$ is arbitrary, Lemma \ref{lem:eder} gives the conclusion.
\end{proof}

Finally, we also require the following approximation result that concerns the marginals.

\begin{proposition}
	\label{prop:approximation_finite_pairs}
	Let $(\mu^k, \nu^k)_{k \in \N}$, $\mu^k \le_{cx} \nu^k$, be a sequence in $\mathcal P_1(\R) \times \mathcal P_1(\R)$ with limit $(\mu, \nu)$ being irreducible.
    For $1 \le j \le J \in \N$, let $(\mu^k_j)_{k \in \N}$ be a convergent sequence in $\mathcal M_1(\R)$ with limit $\mu_j$ and $\sum_{j = 1}^N \mu^k_j = \mu^k$.
    Let $(\nu_j)_{j = 1}^J$, $\mu_j \le_{cx} \nu_j$, be a family in $\mathcal M_1^\ast(\R)$ such that $\nu = \sum_{j = 1}^J \nu_j$. 
    Then, for $1 \le j \le J$, there exist a convergent sequence $(\nu^k_j)_{k \in \N}$ in $\mathcal M_1(\R)$ with limit $\nu_j$ such that
	\begin{equation}
        \mu_j^k \leq_c \nu^k_j\quad\text{and}\quad\sum_{j = 1}^J \nu^k_j = \nu^k.\label{eqnukj}
    \end{equation}
\end{proposition}
The proof of Proposition \ref{prop:approximation_finite_pairs} is rather technical and therefore postponed to Subsection \ref{ssec:proof.prop.approximation_finite_pairs}.
On closer inspection of the statement, this is not completely surprising:
in the setting of Proposition \ref{prop:approximation_finite_pairs}, let $(\mu_j)_{j = 1}^J$ and $(\mu_j^k)_{j = 1}^J$ be families of measures with $\mu_j(\{x_j\}) = \mu_j(\R)$ and $\mu_j^k(\{x_j^k\}) = \mu_j^k(\R)$ for some $x_j, x_j^k \in \R$ so that the points $(x_j)_{j = 1}^J$ are distinct.
For $\pi \in \Pi_M(\mu,\nu)$, we define $\nu_j := \pi_{x_j}$.
Invoking Proposition \ref{prop:approximation_finite_pairs} we obtain $(\nu_j^k)_{j = 1}^J$ and set $\pi^k := \sum_{j = 1}^J \mu_j^k\otimes \nu_j^k$.
Since $\mu_j^k$ is concentrated on a single point and $\mu_j^k \le_{cx} \nu_j^k$, $\pi^k$ defines a martingale coupling in $\Pi_M(\mu^k,\nu^k)$ and, as $\nu^k_j \to \nu_j$ and $\mu_j^k \to \mu_j$ in $\mathcal M_1(\R)$, $(\pi^k)_{k \in \N}$ converges in $\AW_1$ to $\pi$.
Hence, we recover in this particular setting the main result of \cite{BJMP22}, which states that, as long as $\mu^k \le_{cx} \nu^k$, $\mu^k \to \mu$, $\nu^k \to \nu$ in $\mathcal W_1$, any martingale coupling in $\Pi_M(\mu,\nu)$ can be approximated in $\AW_1$ by a sequence $(\pi^k)_{k \in \N}$ with $\pi^k \in \Pi_M(\mu^k,\nu^k)$.

\begin{proof}[Proof of Theorem \ref{thm:approximation}]
    By following the reasoning outlined in \cite[Lemma 5.2]{BJMP22}, incorporating the additional coordinate and replacing \cite[Proposition 2.5]{BJMP22} by Lemma \ref{lem:core_theorem_single_irreducible_to_general}, one can confirm that it suffices to establish the conclusion when $(\bar\mu,\nu)$ is such that $(\proj_1 \bar \mu,\nu)$ is irreducible.
    As the argument runs almost verbatim to the proof of \cite[Lemma 5.2]{BJMP22}, we omit the details and assume from now on that $(\proj_1 \bar \mu,\nu)$ is irreducible.
    
    Let us suppose that $d_\U$ denotes some bounded complete metric compatible with the topology on $\U$ and check that we may suppose w.l.o.g.\ that $\X = \R = \Y$, ${\bar f}(x,u) = 1+|x| + d_\U(u,u_0)$ for some $u_0\in\U$ and $g(y) = 1+|y|$. The convergence of $(\bar\mu^k)_k$ (resp. $(\nu^k)_k$) to $\bar\mu$ in ${\cal P}_{\bar f}(\X\times \U)$ (resp. $\nu$ in ${\cal P}_g(\Y)$) implies that $(\bar\mu^k(\cdot\cap\X\times\U))_k$ (resp. $(\nu^k(\cdot\cap\Y))_k$) converges to $\bar\mu (\cdot\cap\X\times\U)$ in ${\cal P}_{1}(\R\times \U)$ (resp. $\nu(\cdot\cap\Y)$ in ${\cal P}_1(\R)$). We set $\tilde \pi(\cdot)=\pi(\cdot\cap \X\times\U\times\Y)\in \Pi_M(\bar\mu(\cdot\cap\X\times\U),\nu(\cdot\cap\Y))$. Let $(\tilde\pi^k)_k$ be a sequence such that $\tilde\pi^k\in\Pi(\bar\mu^k(\cdot\cap\X\times\U),\nu^k(\cdot\cap\Y))$ and $(\tilde J(\tilde\pi^k)=(x,u,\tilde\pi^k_{x,u})_{\#}\tilde\pi^k)_k$ converges to $\tilde J(\tilde\pi)=(x,u,\tilde\pi_{x,u})_{\#}\tilde\pi$ in $\mathcal P_1(\R \times \U \times \mathcal P_1(\R))$. Since $\X$, $\U$ and $\Y$ are Polish, the Borel sigma-fields satisfy
\begin{align*}
  {\cal B}(\X\times\U\times\Y)&={\cal B}(\X)\otimes {\cal B}(\U)\otimes {\cal B}(\Y)=\sigma\left(\{\{A\cap\X\}\times B\times \{C\cap\Y\}:\;A,C\in{\cal B}(\R),B\in{\cal B}(\U)\}\right)\\&=\left\{D\cap \X\times\U\times\Y:D\in{\cal B}(\R)\otimes{\cal B}(\U)\otimes{\cal B}(\R)\right\}=\left\{D\cap \X\times\U\times\Y:D\in{\cal B}(\R\times\U\times\R)\right\}.
\end{align*}
By Alexandrov's theorem, $\X$ and $\Y$ are countable intersections of open subsets of $\R$. Hence $\X,\Y \in{\cal B}(\R)$ and $\X\times \U\times\Y\in{\cal B}(\R\times \U\times \R)$ so that ${\cal B}(\X)\subset{\cal B}(\R)$, ${\cal B}(\X)\subset{\cal B}(\R)$ and ${\cal B}(\X\times\U\times\Y)\subset{\cal B}(\R\times \U\times \R)$. Let $\pi^k\in \Pi(\bar\mu^k,\nu^k)$ be defined by $\pi^k=\tilde\pi^k|_{{\cal B}(\X\times \U\times\Y)}$. By \cite[Lemma A.7]{BeJoMaPa21b}, the sequence $(J(\pi^k))_k$ is relatively compact in ${\cal P}_{\bar f\oplus \hat g}(\X\times \U\times{\cal P}_{ g}(\Y))$. Let $(J(\pi^{k_j}))_{j}$ denote some subsequence converging to $Q$. Since the injection $i:\X\times\U\times{\cal P}(\Y)\ni(x,u,\rho)\mapsto (x,u,\rho(\cdot\cap\Y))\in\R\times\U\times{\cal P}(\R)
$ is continuous, $i_\#J(\pi^k)=\tilde J(\tilde\pi^k)$ and $i_\#J(\pi)=\tilde J(\tilde\pi)$, we have for any continuous and bounded function $\varphi$ on $\R\times\U\times{\cal P}(\R)$,
\begin{align*}
   Q(\varphi\circ i)=\lim_{j\to\infty}J(\pi^{k_j})(\varphi\circ i)=\lim_{j\to\infty}\tilde J(\tilde \pi^{k_j})(\varphi)=\tilde J(\tilde \pi)(\varphi)=J(\pi)(\varphi\circ i).
\end{align*}
The equality between the left-most and right-most terms remains valid when $\varphi$ is measurable and bounded. Let $A\in{\cal B}(\X)$, $B\in{\cal B}(\U)$ and $C\in{\cal B}({\cal P}(\Y))$. Since ${\cal P}(\R)\ni\rho\mapsto\rho(\Y)$ is measurable, ${\cal P}_{\Y}(\R)=\{\eta\in{\cal P}(\R):\eta(\Y)=1\}$ is a Borel subset of ${\cal P}(\R)$. Since ${\cal P}(\Y)\ni\rho\mapsto \rho(\cdot\cap\Y)\in{\cal P}_{\Y}(\R)$ is an homeomorphism with inverse $r:{\cal P}_{\Y}(\R)\ni\rho\mapsto \rho|_{\cal B(\Y)}\in{\cal P}(\Y)$ when ${\cal P}_{\Y}(\R)$ is endowed with the topology induced by ${\cal P}(\R)$, $r^{-1}(C)\in{\cal B}({\cal P}_{\Y}(\R))\subset{\cal B}({\cal P}(\R))$. For the choice $\varphi=\mathbbm 1_{A\times B\times r^{-1}(C)}$, we deduce that $Q(A\times B\times C)=J(\pi)(A\times B\times C)$. Since ${\cal B}(\X\times\U\times{\cal P}(\Y))={\cal B}(\X)\otimes{\cal B}(\U)\otimes{\cal B}({\cal P}(\Y))$, this implies that $Q=J(\pi)$. Therefore the sequence $(J(\pi^k))_{k}$ converges to $J(\pi)$ in ${\cal P}_{\bar f\oplus \hat g}(\X\times \U\times{\cal P}_{ g}(\Y))$.

    Therefore, we assume from now on that $\X = \R = \Y$, ${\bar f}(x,u) = 1+|x| + d_\U(u,u_0)$ and $g(y) = 1+|y|$.
    
    Moreover, by using Lemma \ref{lem:core_theorem_finite_sum} we may assume that $\pi$ admits the representation \eqref{eq:bar_pi_finite_sum}.
    Let $(\U_j)_{j = 1}^J$ be the associated finite measurable partition of $\U$.
    Without loss of generality, e.g.\ by replacing one element of the partition $\U_k$ such that $\bar \mu(\R \times \U_k) > 0$ with the union of $\U_k$ with all elements $\U_j$ that satisfy $\bar \mu(\R \times \U_j) = 0$ and removing the latter, we can assume that $\min_{1 \le j \le J} \bar \mu(\R \times \U_j) > 0$.
    For $j \in \{1,\ldots,J\}$ and $k \in \N$, we define
	\[
		\bar \mu_j := \mathbbm 1_{\R \times \U_j} \bar \mu,\quad \bar \mu_j^k := \mathbbm 1_{\R \times \U_j} \bar \mu^k, \quad \mu_j := 
                \proj_{1}  \bar \mu_j  \quad\text{and}\quad \mu_j^k := 
                \proj_{1}  \bar \mu_j^k.
	\]
	As $(\U_j)_{j = 1}^J$ is comprised of continuity sets for the first marginal of $\bar\mu$, the weak convergence of $(\bar\mu^k)_{k\in\N}$ to $\bar\mu$ implies that $(\bar\mu^k_j)_{k\in\N}$ converges weakly to $\bar\mu_j$ and, due to the continuity of the first coordinate mapping, $(\mu^k_j)_{k\in\N}$ converges weakly to $\mu_j$ for each $j\in\{1,\cdots,J\}$. 
    All the requirements of Proposition \ref{prop:approximation_finite_pairs} are satisfied, allowing us to identify, for each $j \in \{1,\ldots, J\}$, a sequence of subprobability measures $(\nu^k_j)_{k \in \N}$ such that
	\[
	\nu^k_j\underset{k\to+\infty}{\longrightarrow}\nu_j, \quad \mu_j^k \leq_c \nu^k_j\quad\text{and}\quad\sum_{j = 1}^J \nu^k_j = \nu^k.
	\]
	From now on we will assume that $k$ is large enough so that $\min_{1\le j\le J}\mu^k_j(\R) > 0$.
    Weak convergence of the original sequences yields, for each $j\in\{1,\cdots,J\}$, that
	the normalized sequence $(\bar\mu^k_j/\mu^k_j(\R))_{k \in \N}$ (resp. $(\nu^k_j/\mu^k_j(\R))_{k \in \N}$) converges weakly to $\bar\mu_j/\mu_j(\R)$ (resp. $\nu_j/\mu_j(\R)$) as $k\to\infty$.
	As $(\bar\mu^k)_{k \in \N}$ and $(\nu^k)_{k \in \N}$ are $\W_1$-convergent sequences, it then follows easily from Lemmas \ref{lem:convergence_subprobabilities} and \ref{lem:convergence_dominated} that the normalized sequences converge in $\W_1$. 
	Thus, we can apply \cite[Theorem 2.6]{BJMP22} and obtain an $\AW_1$-convergent sequence $(\gamma^k_j)_{k \in \N}$ of martingale couplings with limit $\gamma_j$ where
    \[ 
    \gamma_j := \frac{\mu_j \times K_{j,x}}{\mu_j(\R)}\in \Pi_M\left(\frac{\mu_j}{\mu_j(\R)},\frac{\nu_j}{\mu_j(\R)}\right)
        \quad\text{and}\quad
            \gamma^k_j \in \Pi_M\bigg(\frac{\mu^k_j}{\mu^k_j(\R)},\frac{\nu^k_j}{\mu^k_j(\R)}\bigg).
    \]
	Further, write
	\[
		\bar \gamma^k_j := \frac{\bar\mu^k_j\times \gamma^k_{j,x}}{\mu^k_j(\R)}
		\quad\mbox{and}\quad
		\bar \gamma_j := \frac{ \mathbbm 1_{\R \times \U_j \times \R} \pi}{\mu_j(\R)}=\frac{\bar\mu_j\times K_{j,x}}{\mu_j(\R)}.
	\]
	To prove that $(\bar \gamma^k_j)_{k \in \N}$ converges in $\AW_1$ to $\bar \gamma_j$, we choose
    \[ 
        \chi^k_j \in \Pi \left( \frac{\bar\mu^k_j}{\mu^k_j(\R)}, \frac{\bar \mu_j}{\mu_j(\R)} \right)
        \quad\text{resp.}\quad
         \hat \chi^k_j\in \Pi\left(\frac{\mu^k_j}{\mu^k_j(\R)},\frac{\mu_j}{\mu_j(\R)}\right)
    \]
    $\W_1$-optimal between its marginals resp.\ optimal for $\AW_1(\gamma^k_j,\gamma_j)$.
	For ease of notation, we moreover define $\check \chi^k_j$ as $\check \chi^k_j(dx,du,d\hat x,d\hat u, dz) := \chi^k_j(dx,du,d\hat x,d\hat u) \hat \chi^k_{j,x}(dz)$ and compute
	\begin{align}
		\AW_1(\bar \gamma^k_j, \bar \gamma_j) &\le
		\W_1\bigg(\frac{\bar \mu^k_j}{\mu^k_j(\R)},\frac{\bar \mu_j}{\mu_j(\R)}\bigg) + \int \W_1(\gamma^k_{j,x}, K_{j,\hat x}) \, \chi^k_j(dx,du,d\hat x,d\hat u)
		\notag\\
		&\le \W_1\bigg(\frac{\bar \mu^k_j}{\mu^k_j(\R)},\frac{\bar \mu_j}{\mu_j(\R)}\bigg) + \int \left(\W_1(\gamma^k_{j,x}, K_{j,z}) + \W_1(K_{j,z}, K_{j,\hat x}) \right)\, \check \chi^k_j(dx,du,d\hat x,d\hat u,d z)
		\notag\\
		&\le
		\W_1\bigg(\frac{\bar \mu^k_j}{\mu^k_j(\R)},\frac{\bar \mu_j}{\mu_j(\R)}\bigg)+ \AW_1(\gamma^k_j,\gamma_j) + \int \W_1(K_{j,z}, K_{j,\hat x}) \, \check \chi^k_j(dx,du,d\hat x, d\hat u, dz),
           \label{eqawgamkjgamj}
	\end{align}
        By Lemma \ref{lem:eder} due to Eder \cite{Ed19} and since \[ \int |z-\hat x| \check \chi^k_j(dx,du,d\hat x,d\hat u,d z)\le \W_1\bigg(\frac{\bar \mu^k_j}{\mu^k_j(\R)},\frac{\bar \mu_j}{\mu_j(\R)}\bigg)+ \AW_1(\gamma^k_j,\gamma_j)\underset{k\to\infty}{\longrightarrow}0, \]
        the right-hand side of \eqref{eqawgamkjgamj} goes to $0$ as $k\to\infty$.

    In the next step, we revert the normalization by setting
    \[
    	\pi^k := \sum_{j=1}^J\mu^k_j(\R)\bar\gamma^k_j\in\bar\Pi_M(\bar\mu^k,\nu^k).
    \]
    Let $\varepsilon\in (0,\min_{1\le j\le J}\mu_j(\R))$ and assume that $k$ is sufficiently large so that $\max_{1\le j\le J}|\mu^k_j(\R)-\mu_j(\R)|\le\varepsilon$.
    We split each of $\bar \pi^k$ and $\bar \pi$ into two parts:
    \[
    	\pi^k=\sum_{j=1}^J\left(\mu_j(\R)-\varepsilon\right)\bar\gamma^k_j
    	+\sum_{j=1}^J \left(\mu^k_j(\R)-\mu_j(\R)+\varepsilon\right)\bar\gamma^k_j
    	\mbox{ and }
    	\pi=\sum_{j=1}^J \left(\mu_j(\R)-\varepsilon\right)\bar\gamma_j+\varepsilon\sum_{j=1}^J\bar\gamma_j.
    \]
    Because $(\bar \mu_j)_{j = 1}^J$ are pairwise singular, we can apply \cite[Lemma 3.7]{BJMP22} to deduce that
    \[
    	\lim_{k \to \infty} \AW_1 \left( \sum_{j=1}^J\left(\mu_j(\R)-\varepsilon\right)\bar\gamma^k_j, \sum_{j=1}^J \left(\mu_j(\R)-\varepsilon\right)\bar\gamma_j \right) = 0.
    \]
    With the help of \cite[Lemma 3.6]{BJMP22} and \cite[Lemma 3.1 (a)(c)]{BJMP22}, we conclude that
    \begin{equation}
    	\label{eq:thm_core_theorem.convergence}
    	\limsup_{k \to \infty} 
    	\AW_1 \left( \pi^k, \pi \right) \le C \left( I^1_{J\epsilon}(\bar \mu) + I^1_{J\epsilon}(\nu) \right),
    \end{equation}
    where $C > 0$ is a constant that does not depend on $(k,\epsilon)$ and, for a Polish space $\X$, $\delta > 0$, $x_0 \in \X$, and $\eta \in \mathcal M_1(\X)$, $I^1_\delta(\eta)$ is given by
    \[
    	I^1_\delta(\eta) := \sup_{\tau \in \mathcal M_1(\X), \tau \le \eta,\tau(\X)\le\delta}\int_\X d_\X (x,x_0) \, \tau(dx).
    \]
    As $\epsilon > 0$ was arbitrary and we have by \cite[Lemma 3.1 (b)]{BJMP22} that $\lim_{\epsilon\searrow 0}\left( I^1_{J\epsilon}(\bar \mu) + I^1_{J\epsilon}(\nu) \right)=0$, we can infer from \eqref{eq:thm_core_theorem.convergence} that $(\bar \pi^k)_{k \in \N}$ converges to $\bar \pi$ in $\AW_1$.
\end{proof}
\subsection{Proofs of Corollary \ref{cor:intro.approximation.process} and Propositions \ref{prop:stability.VIX} and \ref{prop:intro.stability.process}
} \label{ssec:examples}
We are first going to prove the following stronger variants of Corollary \ref{cor:intro.approximation.process} and Proposition \ref{prop:intro.stability.process} before deducing Proposition \ref{prop:stability.VIX}. 
Let $f\colon \X \to [1,+\infty)$ be a continuous growth function such that 
\[ \liminf_{ \substack{ |x| \to \infty \\ x \in \X} } \frac{f(x)}{|x|} > 0. \]
The topological space $\mathcal P_f(\X)$ is defined like $\mathcal P_g(\Y)$ with $\X$ and $f$ replacing $\Y$ and $g$. The topological space $\mathcal P_{f \oplus \hat g}(\X \times \mathcal P_g(\Y))$ is defined analogously to $\mathcal P_{\bar f \oplus \hat g}(\X \times\U\times \mathcal P_g(\Y))$ but without the $u$ coordinate.

\begin{corollary}
      \label{cor:approximation.process}
      Let $(\mu^k,\nu^k)_{k \in \N}$, $\mu^k \le_{cx} \nu^k$, be a convergent sequence in $\mathcal P_f(\X)\times\mathcal P_g(\Y)$ with limit $(\mu,\nu)$.
      Then, every $P\in\Lambda_M(\mu,\nu)$ is the limit in $\mathcal P_{f \oplus \hat g}(\X \times \mathcal P_g(\Y))$ of a sequence $(P^k)_{k \in \N}$ with $P^k\in\Lambda_M(\mu^k,\nu^k)$.
\end{corollary}

\begin{proof}
    Let $P=\mu(dx)P_x(d\rho)\in\Lambda_M(\mu,\nu)$. By Lemma 3.22 \cite{kallenberg2ndedition}, there exists a measurable mapping $\R\times[0,1]\ni(x,u)\mapsto\pi_{x,u}\in{\cal P}_1(\R)$ such that ${\pi_{x,u}}_{\#}\Leb(du)=P_x$. 
    Then for $\bar\mu = \mu \otimes \Leb$, $\pi(dx,du,dy)=\bar\mu(dx,du)\pi_{x,u}(dy)\in\Pi_M(\bar\mu,\nu)$.
    The sequence $(\bar\mu^k = \mu^k \otimes \Leb)_{k\in\N}$ converges to $\bar\mu$ in $\mathcal P_{\bar f}(\X\times[0,1])$ where $\bar f(x,u)=f(x)$. By Theorem \ref{thm:approximation}, $\pi$ is the limit in the adapted weak topology of a sequence $(\pi^k)_{k \in \N}$ with $\pi^k \in \Pi_M(\bar \mu^k,\nu^k)$. 
    Therefore, we have $P^k =J(\pi^k) \in\Lambda_M(\mu^k,\nu^k)$ and get that $(P^k)_{k\in\N}$ converges to $P$ in $\mathcal P_{f \oplus \hat g}(\X \times \mathcal P_g(\Y))$.
\end{proof}

\begin{proposition} \label{thm:wmot.stabilitys}
    Assume that $C \colon \X \times \mathcal P_g(\Y) \to \R$ is continuous and that there is a constant $K > 0$ such that, for all $(x,\rho) \in \X \times \mathcal P_g(\Y)$,
    \[
        |C(x,\rho)| \le K \left( 1 + f(x) + \rho(g) \right).
    \]
    Then, the value $\hat V_C$ is attained and continuous on $\{ (\mu,\nu) \in \mathcal P_f(\X)\times\mathcal P_g(\Y) : \mu \le_{cx} \nu \}$. Moreover, if $(\mu^k,\nu^k)_{k\in\N}$ with $\mu^k\le_{cx}\nu^k$  is convergent in $\mathcal P_f(\X)\times\mathcal P_g(\Y)$ and $P^k \in \Lambda_M(\mu^k,\nu^k)$ is a sequence of optimizers of \eqref{eq:WMOT}, then so are its accumulation points.
    \end{proposition}

      \begin{proof}
  Note that the mapping $\mathcal P_{f \oplus \hat g}(\X \times \mathcal P_g(\Y))\ni P\mapsto P(C)$ is continuous.   
  Using \cite[Lemma A.7]{BeJoMaPa21b} for the relative compactness, we easily check that for $(\mu,\nu)\in \mathcal P_f(\X)\times\mathcal P_g(\Y)$ with $\mu\le_{cx}\nu$, $\Lambda_M(\mu,\nu)$ is compact and non-empty. Therefore $\hat V_C$ is attained.
  
  Let $(\mu^k,\nu^k)_{k\in\N}$ with $\mu^k\le_{cx}\nu^k$ be convergent in $\mathcal P_f(\X)\times\mathcal P_g(\Y)$ with limit $(\mu,\nu)$. 
  Let $P^\star \in\Lambda_M(\mu,\nu)$ be optimal for $\hat V_C(\mu,\nu)$. By Corollary \ref{cor:approximation.process}, there exists a sequence $(P^k)_{k \in \N}$ with $P^k\in\Lambda_M(\mu^k,\nu^k)$ that converges to $P^\star$ in $\mathcal P_{f \oplus \hat g}(\X \times \mathcal P_g(\Y))$. Therefore \begin{align}
   \hat V_C(\mu,\nu)=P^\star(C)=\lim_{k\to \infty}P^k(C)\le \liminf_{k\to\infty}\hat V_C(\mu^k,\nu^k).\label{eq:uscvhc}
\end{align}
Now choosing $P^{k,\star}\in\Lambda_M(\mu^k,\nu^k)$ such that $\hat V_C(\mu^k,\nu^k)=P^{k,\star}(C)$, we may extract a subsequence $(P^{k_j,\star})_{j\in\N}$ converging to some $P^\infty$ in $\mathcal P_{f \oplus \hat g}(\X \times \mathcal P_g(\Y))$ and such that $\lim_{j\to\infty}P^{k_j,\star}(C)=\limsup_{k\to\infty}\hat V_C(\mu^k,\nu^k)$. Then $P^\infty\in\Lambda_M(\mu,\nu)$ and
\[ \hat V_C(\mu,\nu)\ge P^\infty(C)=\lim_{j\to\infty}P^{k_j}(C)=\limsup_{k\to\infty}\hat V_C(\mu^k,\nu^k). \]
With \eqref{eq:uscvhc}, we deduce that $\hat V_C(\mu,\nu)=\lim_{k\to\infty}\hat V_C(\mu^k,\nu^k)$ so that $\hat V_C$ is continuous on $\{ (\eta, \theta) : \eta \le_{cx} \theta \} \subseteq \mathcal P_f(\X) \times \mathcal P_g(\Y)$. Moreover $\hat V_C(\mu,\nu)= P^\infty(C)$ and this equality remains true for any accumulation point $P^\infty$ of $(P^{k,\star})_{k\in\N}$.
\end{proof}
\begin{proof}[Proof of Proposition \ref{prop:stability.VIX}]
    Set $\X=(0,+\infty)$, $f:\X\ni x\mapsto |\ln(x)| + |x|$ and $C_{\rm VIX}:\X\times{\cal P}_f(\X)\ni(x,\rho)\mapsto -\sqrt{\rho(\ell_x)\vee 0}$ where we recall that $\ell_x=\frac{2}{T_2 - T_1}\ln(x/y)$ for $(x,y)\in\X^2$. Let $\mu,\nu\in{\cal P}_f(\X)$. For $\pi \in \Pi_{\rm VIX}(\mu,\nu)$, we have that $(x,\pi_{x,u})_\# \pi = P \in \Lambda_M(\mu,\nu)$ with 
\begin{equation}
    \label{eq:from.VIX.to.WMOT}
    \int \sqrt{\pi_{x,u}(\ell_x)} \, \proj_{1,2}\pi(dx,du) = -P(C_{\rm VIX}).
\end{equation}
Therefore $D_{\rm sub}(\mu,\nu)\ge -\hat V_{C_{\rm VIX}}(\mu,\nu)$.

On the other hand, if $P \in \Lambda_M(\mu,\nu)$ then $\pi := (x,\sqrt{\rho(\ell_x)},y)_\# P(dx,d\rho) \rho(dy) \in \Pi_{\rm VIX}(\mu,\nu)$ since $\rho(\ell_x)\in[0,+\infty)$ $P(dx,d\rho)$ a.e. by Jensen's inequality combined with $\int \rho(f)P(dx,d\rho)=\nu(f)<\infty$ and, for $\varphi:\X\times \R_+\to\R$ measurable and bounded,
\begin{align*}
    \int_{\X\times\R_+\times\X}\varphi(x,u)(u^2-\ell_x(y))\pi(dx,du,dy)&=\int_{{\cal P}_{f\oplus \hat f(\X\times{\cal P}_f(\X))}}\varphi(x,\sqrt{\rho(\ell_x)})\int_\X(\rho(\ell_x)-\ell_x(y))\rho(dy)P(dx,d\rho)\\&=0
.\end{align*}Therefore \eqref{eq:from.VIX.to.WMOT} again holds.
We conclude that $D_{\rm sub}(\mu,\nu) = -\hat V_{C_{\rm VIX}}(\mu,\nu)$ and deduce from Proposition \ref{thm:wmot.stabilitys} applied with $\Y = \X = (0,\infty)$ and $g=f$ the continuity of $D_{\rm sub}$.
\end{proof}

\subsection{Stability of extended weak martingale optimal transport problems: proof of Theorem \ref{thm:intro.stability}}
\label{ssec:stability}

This section is dedicated to the proof of a stronger variant of Theorem \ref{thm:intro.stability}, that is Theorem \ref{thm:stability} below.

\begin{assumpB} \label{ass:cost.B}
    We say that a cost function $C \colon \X \times \U \times \mathcal P_g(\Y) \to \R$ satisfies Assumption \ref{ass:cost.B} if $C$ is continuous and there is a constant $K > 0$ such that, for all $(x,u,\rho) \in \X \times \U \times \mathcal P_g(\Y)$,
    \begin{equation}
        \label{eq:assB.cost.bound}
        |C(x,u,\rho)| \le K \left( 1 + {\bar f}(x,u) + \rho(g) \right).
    \end{equation}
\end{assumpB}

\begin{theorem} \label{thm:stability}
    Let $C$ satisfy Assumption \ref{ass:cost.B} and $C(x,u,\cdot)$ be convex. 
    Then the value function $V_C$ is attained and continuous on $\{ (\bar \mu,\nu) : \proj_1\bar\mu\le_{cx}\nu\} \subseteq \mathcal P_{\bar f}(\X \times \U) \times \mathcal P_g(\Y)$.
    
    Furthermore, when $(\bar\mu^k,\nu^k)_{k \in \N}$, $\proj_1 \bar \mu^k \le_{cx} \nu^k$, converges to $(\bar \mu,\nu)$ in $\mathcal P_{\bar f}(\X \times \U) \times \mathcal P_g(\Y)$, we have:
    \begin{enumerate}[label = (\roman*)]
        \item \label{it:stability.1} if $(\pi^k)_{k \in \N}$ is a sequence of optimizers $\pi^k \in \Pi_M(\bar \mu^k,\nu^k)$ of \eqref{eq:pwmot}, so are its accumulation points;
        \item \label{it:stability.2} if $C(x,u,\cdot)$ is strictly convex, then optimizers of \eqref{eq:WMOT} are unique and $(\pi^k)_{k \in \N}$ converges to the optimizer of \eqref{eq:pwmot} with marginals $(\bar \mu,\nu)$ in the adapted weak topology.
    \end{enumerate}
\end{theorem}

\begin{proof}
  By \cite[Proposition A.12 (b)]{BeJoMaPa21b}, the map $\mathcal P_{{\bar f}\oplus g}(\X\times\U\times\Y)\ni \pi \mapsto \int C(x,u,\pi_{x,u}) \, \proj_{1,2} \pi(dx,du)$ is lower semicontinuous. Since $\Pi_M(\bar\mu,\nu)$ is a compact subset of $\mathcal P_{{\bar f}\oplus g}(\X\times\U\times\Y)$, we deduce that the value function is attained.

  Let $((\bar\mu^k,\nu^k))_{k\in\N}$ in $\{ (\bar \eta, \theta) : \proj_1 \bar \eta \le_{cx} \theta \} \subseteq \mathcal P_{\bar f}(\X \times \U) \times \mathcal P_g(\Y)$ be convergent with limit $(\bar \mu,\nu)$. Let $\pi^\star\in \Pi^M(\bar \mu,\nu)$ be such that $V_C(\bar\mu,\nu)=J(\pi^\star)(C)$. By Theorem \ref{thm:approximation}, there exists a sequence $(\pi^k)_{k \in \N}$ with $\pi^k \in \Pi_M(\bar \mu^k,\nu^k)$ such that $J(\pi^k) \to J(\pi^\star)$ in $\mathcal P_{{\bar f} \oplus \hat g}(\X \times \U \times \mathcal P_g(\Y))$.
 Since the mapping $\mathcal P_{{\bar f} \oplus \hat g}(\X \times \U\times \mathcal P_g(\Y))\ni P\mapsto P(C)$ is continuous, we deduce that
   \begin{equation}
      \limsup_{k \to \infty} V_C(\bar \mu^k,\nu^k) \le \lim_{k \to \infty} J(\pi^k)(C) = J(\pi^\star)(C)=V_C(\mu,\nu).\label{eq:uscvc}
   \end{equation}
 Choosing now $\pi^{k,\star}\in \Pi_M(\bar\mu^k,\nu^k)$ such that $V_C(\bar\mu^k,\nu^k)=\int C(x,u,\pi^{k,\star}_{x,u}) \pi^{k,\star}(dx,du)$, we may extract a subsequence $(\pi^{k_j,\star})_{j\in\N}$ such that $(\pi^{k_j,\star})_{j\in\N}$ converges to $\pi^\infty$ in $\mathcal P_{{\bar f}\oplus g}(\X\times\U\times\Y)$ and $\lim_{j\to\infty}\int C(x,u,\pi^{k_j,\star}_{x,u}) \bar\mu^{k_j}(dx,du)=\liminf_{k\to\infty}V_C(\mu^k,\nu^k)$. The limit $\pi^\infty$ belongs to $\Pi_M(\bar\mu,\nu)$ and, by lower semicontinuity of $\mathcal P_{{\bar f}\oplus g}(\X\times\U\times\Y)\ni \pi \mapsto \int C(x,u,\pi_{x,u}) \, \proj_{1,2} \pi(dx,du)$, we deduce that
  $$V_C(\bar\mu,\nu)\le \int C(x,u,\pi^\infty_{x,u}) \bar\mu(dx,du)\le \lim_{j\to\infty}\int C(x,u,\pi^{k_j,\star}_{x,u}) \bar\mu^{k_j}(dx,du)
  =\liminf_{k\to\infty}V_C(\mu^k,\nu^k).$$
With \eqref{eq:uscvc}, we deduce that $\lim_{k\to\infty}V_C(\bar\mu^k,\nu^k)=V_C(\mu,\nu)$ so that $V_C$ is continuous on $\{ (\bar \eta, \theta) : \proj_1 \bar \eta \le_{cx} \theta \} \subseteq \mathcal P_{\bar f}(\X \times \U) \times \mathcal P_g(\Y)$. Moreover, $V_C(\bar\mu,\nu)= \int C(x,u,\pi^\infty_{x,u}) \bar\mu(dx,du)$ and this equality remains true for any accumulation point $\pi^\infty$ of $(\pi^{k,\star})_{k\in\N}$ in $\mathcal P_{{\bar f}\oplus g}(\X\times\U\times\Y)$.


    To show \ref{it:stability.2}, we assume the opposite, that is, that $(\pi^k)_{k \in \N}$ admits a subsequence which does not have $\pi^\star$ as an accumulation point w.r.t.\ the adapted weak topology.
    By \cite[Lemma A.7]{BeJoMaPa21b}, this particular subsequence admits a subsequence $(\pi^{k_j})_{j \in \N}$ such that $(J(\pi^{k_j}))_{j \in \N}$ converges in $\mathcal P_{{\bar f} \oplus \hat g}(\X \times \U \times \mathcal P_g(\Y))$ to $P$.
    We define $\tilde \pi \in \Pi_M(\bar\mu,\nu)$ by $\tilde \pi = \bar \mu \times \tilde \pi_{x,u}$ with $\tilde \pi_{x,u} = \int \rho(dy) \, P_{x,u}(d\rho)$.
    As $C(x,u,\cdot)$ is convex and continuous, we have by Jensen's inequality
    \[
        \int C(x,u,\tilde \pi_{x,u}) \, \bar\mu(dx,du) \le \int C(x,u,\rho) \, P(dx,du,d\rho) = \lim_{k \to \infty} \int C(x,u,\pi^k_{x,u}) \, \bar\mu^k(dx,du) = V_C(\bar\mu,\nu).
    \]
    
    In particular, $\tilde \pi$ is an optimizer of $V_C(\bar\mu,\nu)$ and, by strict convexity of $C(x,u,\cdot)$, we have $J(\tilde \pi) = P$ and uniqueness of optimizers.
    Thus, $\tilde \pi = \pi^\star$, and we also get $J(\pi^\star) = P$.
    Hence, $(\pi^{k_j})_{j \in \N}$ converges in the adapted weak topology to $\pi^\star$, which is a contradiction and completes the proof.
\end{proof}

\subsection{Stability of the shadow couplings: proof of Proposition \ref{prop:shadcoupl}} \label{ssec:SC}
Let us first state a consequence of Proposition \ref{prop:shadcoupl} concerning the shadow couplings.

In view of Sklar's theorem, it is natural to parametrize the dependence structure between $\mu$ and the Lebesgue measure on $[0,1]$ in the lift $\bar \mu \in \Pi(\mu,\Leb)$ of $\mu$ by copulas i.e. probability measures on $[0,1]\times[0,1]$ with both marginals equal to the Lebesgue measure. We call shadow coupling between $\mu$ and $\nu$ with copula $\chi$ the shadow coupling between $\mu$ and $\nu$ with source equal to the image $\bar\mu_\chi$ of $\chi$ by $[0,1]\times[0,1]\ni(v,u)\mapsto (F_\mu^{-1}(v),u)\in\R\times [0,1]$, where $F_\mu^{-1}$ denotes the quantile function of $\mu$. \begin{corollary}\label{cor:contshc}
     The shadow coupling with copula $\chi$ is continuous on the domain $\{ (\mu,\nu) : \mu \le_{cx} \nu\} \subseteq \mathcal P_p(\R ) \times \mathcal P_p(\R)$ and with range $(\mathcal P_p(\R \times \R),\W_p)$ and even continuous in $\AW_p$ at each couple $(\mu,\nu)$ such that $\mu$ does not weight points.
 \end{corollary}
 For the Hoeffding-Fréchet copula, $\chi(dv,du)=\Leb(du)\delta_u(dv)$, we recover the stability w.r.t. the marginals $\mu$ and $\nu$ of the left-curtain coupling proved by Juillet in \cite{Ju16}. For the independence copula $\chi(dv,du)=\Leb(dv)\otimes\Leb(du)$, we deduce the continuity of the sunset coupling. 

The proof that the selector $SC$ of the lifted shadow coupling is continuous when the codomain $\mathcal P_p(\R \times [0,1] \times \R)$ is endowed with the adapted Wasserstein distance $\AW_p$ relies on the fact that, by \eqref{eq:SC.extremal}, the selector $\SC$ takes values in the following extremal set of extended martingale couplings
\[  
    \Pi_{M,p}^{\rm ext} := \{ \pi \in \mathcal P_p(\R \times\U \times \R ) : \# \supp(\pi_{x,u}) \in \{1,2\} \text{ and } \mean(\pi_{x,u}) = x \text{ $\pi$-a.s.}\}.
\]
The set $\Pi_{M,p}^{\rm ext}$ is extremal in the following sense:
when $\pi \in \Pi_{M,p}^{\rm ext}$ and $P \in \mathcal P_p(\R \times \U \times \mathcal P_p(\R))$ with $I(P) = \pi$, where $I(P)$ is the unique measure that satisfies
\[
    \int f(x,u,y) \, I(P)(dx,du,dy) = \int \int f(x,u,y) \, \rho(dy) \, P(dx,du,d\rho),
\]
for all $f \in C_b(\R \times \U \times \R)$ and $\mean(\rho) = x$ $P$-a.s., then we already have $P = J(\pi)$.
Proceeding from this observation, the next lemma shows that on $\Pi_{M,p}^{\rm ext}$ the $p$-Wasserstein topology coincides with the $p$-adapted Wasserstein topology, which we in turn use to prove Proposition \ref{prop:shadcoupl}.

\begin{lemma}\label{lem:ext.WAW}
    The identity map $\id$ on $\mathcal P_p(\R \times \U \times \R)$ is $(\mathcal W_p,\mathcal{AW}_p)$-continuous at any $P\in \Pi^{\rm ext}_{M,p}$.
    In particular, the metric spaces $(\Pi_{M,p}^{\rm ext}, \mathcal W_p)$ and $(\Pi_{M,p}^{\rm ext}, \mathcal{AW}_p)$ are topologically equivalent.
\end{lemma}

\begin{proof}
    We follow a similar line of reasoning as used in \cite[Lemma 7]{Pa22}.
    As $\W_p \le \AW_p$, it suffices to show that, given a sequence $(\pi^k)_{k \in \N}$ in $\mathcal P_p(\R \times \U \times \R)$ with $\mean(\pi^k_{x,u}) = x$ $\pi^k$-a.s.\ and $\pi \in \Pi_{M,p}^{\rm ext}$,
    \[
        \lim_{k \to \infty} \W_p(\pi^k,\pi) = 0 \implies
        \lim_{k \to \infty} \AW_p(\pi^k,\pi) = 0.
    \]
    So, let $(\pi^k)_{k \in \N}$ and $\pi$ be as above and assume that $\pi^k \to \pi$ in $\mathcal W_p$.
    Observe that $J|_{\Pi_{M,p}^{\rm ext}}$ is bijective onto
    \begin{equation}
        \label{eq:PiM.extremal}
        J(\Pi_{M,p}^{\rm ext}) = \{ P \in \mathcal P_p(\R\times\U\times\mathcal P_p(\R)) : I(P)  \in \Pi_{M,p}^{\rm ext} \text{ and }\mean(\rho) = x \text{ $P(dx,du,d\rho)$-a.e.} \},  
    \end{equation}
    with inverse $I$. Using \cite[Lemma 2.3]{BaBePa18}, we find that the sequence $(J(\pi^k))_{k \in \N}$ is $\mathcal W_p$-relatively compact in $\mathcal P_p(\R\times\U\times \mathcal P_p(\R))$.
    Therefore, there is a subsequence $(\pi^{k_j})_{j \in \N}$ such that $J(\pi^{k_j}) \to P$.
    Since $\pi^{k_j} \to I(P) = \pi \in \Pi^{\rm ext}_M$ and $\mean(\rho) = x$ $P(dx,du,d\rho)$-a.e., we get by \eqref{eq:PiM.extremal} that $P \in J(\Pi_{M,p}^{\rm ext})$ which yields by bijectivity of $J|_{\Pi_{M,p}^{\rm ext}}$ that $P = J(\pi)$.
    Hence, $J(\pi^{k_j}) \to J(\pi)$ in $\W_p$ which means that $\pi^{k_j} \to \pi$ in $\AW_p$.

    Since any subsequence of $(\pi^k)_{k \in \N}$ admits by above reasoning an $\AW_p$-convergent subsequence with limit $\pi$, we conclude that $\pi^k \to \pi$ in $\AW_p$.
\end{proof}
The proof of Proposition \ref{prop:shadcoupl} also relies on the following two lemmas, the proof of which are postponed to the end of the current section.
\begin{lemma}\label{lem:MZ}
   Let $\cal V,\cal Z$ be Polish spaces, $(\theta^k)_{k \in \N}$ be a sequence in $\mathcal P({\cal V})$ that converges in total variation to $\theta$, and let $\varphi^k\colon{\cal V}\to {\cal Z}\;k\in\N$, and $\varphi\colon{\cal V}\to {\cal Z}$ be measurable functions.
   Then 
   \[ (\id, \varphi^k)_\# \theta^k \to (\id,\varphi)_\# \theta \text{ in }\mathcal{P}({\cal V} \times {\cal Z}) \implies \varphi^k \to \varphi \text{ in }\theta\text{-probability}. \]
 \end{lemma}\begin{lemma}\label{lem:w1bin}
   Let $x,y,z \in \R$ with $y<x<z$, and $((y^k,z^k))_{k\in\N}$ be a $(-\infty,x]\times[x,+\infty)$-valued sequence such that for each $k$, either $y^k<x<z^k$ or $y^k=x=z^k$.
   Then we have
   \begin{enumerate}[label = (\roman*)]
       \item \label{it:w1bin.1} ${\cal W}_1(B(x,y^k,z^k), B(x,y,z))\to 0\iff |y^k-y|+|z^k-z|\to 0$,
       \item \label{it:w1bin.2} ${\cal W}_1(B(x,y^k,z^k), \delta_x)\to 0\iff (z^k-x)\wedge(x-y^k)\to 0$.
   \end{enumerate}
\end{lemma}
\begin{proof}[Proof of Proposition \ref{prop:shadcoupl}]
As optimizers of $V_{\rm SC}$ are unique, we immediately obtain from Theorem \ref{thm:stability} applied with $C(x,u,\rho)=\int_\R (1-u)\sqrt{1+y^2}\rho(dy)$ continuity of 
\begin{align}
    \label{eq:SC.value.cont}
    V_{\rm SC} \colon& \{ (\bar\mu,\nu) \in \mathcal P_p(\R \times [0,1]) \times \mathcal P_p(\R) : \proj_1 \bar \mu \le_{cx} \nu, \proj_2 \bar\mu = \Leb \} \to \R, \\
    \label{eq:SC.selector.cont}
    \SC \colon& \{ (\bar\mu,\nu) \in \mathcal P_p(\R \times [0,1]) \times \mathcal P_p(\R) : \proj_1 \bar \mu \le_{cx} \nu, \proj_2 \bar\mu = \Leb \} \to \mathcal P_p(\R \times [0,1] \times \R),
\end{align}
when the domain is endowed with the product of the corresponding Wasserstein $p$-topologies.
Since $\SC$ is a continuous function taking values in $\Pi^{\rm ext}_M$, Lemma \ref{lem:ext.WAW} ensures that it is still continuous when the codomain is endowed with the stronger $\AW_p$-distance.
Therefore when 
By Proposition \ref{prop:shadcoupl} we have that
\[
    \SC(\bar\mu^k,\nu^k) \to \SC(\bar\mu,\nu) \quad \text{in }\AW_1,
\]
which is equivalent to $\W_1$-convergence of
\[
    (\id, B(X,T_1^k,T_2^k))_\# \bar\mu^k = J(\SC(\bar\mu^k,\nu^k)) \to J(\SC(\bar\mu,\nu)) = (\id, B(X,T_1,T_2))_\# \bar\mu,
\]
where $X \colon \R \times [0,1]\ni(x,u) \mapsto x\in\R$.
Applying Lemma \ref{lem:MZ} in the setting
\begin{gather*}
    {\cal V}=\R\times[0,1], \quad {\cal Z}={\cal P}_1(\R),\quad \theta^k=\bar\mu^k, \quad \theta=\bar\mu,\quad
     \varphi^k =B(X,T^k_1,T^k_2)\quad\text{and}\quad \varphi=B(X,T_1,T_2),
\end{gather*} 
yields $B(X,T^k_1,T^k_2) \to B(X,T_1,T_2)$ in $\bar\mu$-probability.
There exists a subsequence such that this convergence holds $\bar\mu$-a.s.
Hence, we can invoke Lemma \ref{lem:w1bin} and derive the assertion in the second statement of the proposition for this particular subsequence.
By the above reasoning any subsequence admits a subsubsequence which fulfills the conclusion of the second statement of the proposition, which readily implies the statement.\end{proof}
 \begin{proof}[Proof of Corollary \ref{cor:contshc}]
 For the continuity in $\W_p$, it is enough to combine Proposition \ref{prop:shadcoupl} with \begin{align*}
   &\forall \mu,\mu'\in{\cal P}_p(\R),\;\W^p_p(\bar\mu_\chi,\bar\mu'_\chi)\le \int_{[0,1]\times[0,1]}|F_\mu^{-1}(v)-F_{\mu'}^{-1}(v)|^p\chi(dv,du)=\W_p^p(\mu,\mu'),\\
 &\forall \pi,\pi'\in\mathcal P_p(\R \times [0,1] \times \R),\;\W_p(\proj_{1,3}\pi,\proj_{1,3}\pi')\le \W_p(\pi,\pi')\le \AW_p(\pi,\pi').\end{align*} To prove the reinforced continuity in $\AW_p$, we consider a sequence $((\mu^k,\nu^k)_k)$ in $\mathcal P_p(\R ) \times \mathcal P_p(\R)$ with $\mu^k\le_{cx}\nu^k$ converging to $(\mu,\nu)$ where $\mu$ does not weight points. For notational simplicity, we denote $\SC^k$ and $\SC$ respectively in place of $\SC(\bar\mu^k_\chi,\nu^k)$ and $\SC(\bar\mu_\chi,\nu)$. By the reinforcement of Proposition \ref{prop:shadcoupl}, $\AW_p(\SC^k
     ,\SC)\rightarrow 0$. Let $\eta^k\in\Pi(\bar\mu^k_\chi,\bar\mu_\chi)$ be optimal for $\AW_p(\SC^k,\SC)$. We have
\begin{align*}
   \int_{[0,1]\times[0,1]}&
  \W_p^p(\SC^k_{F_{\mu^k}^{-1}(v),u},\SC_{F_{\mu}^{-1}(v),u})\chi(dv,du)\le 2^{p-1}\AW^p_p(\SC^k,\SC)\\&+2^{p-1}\int_{[0,1]\times[0,1]\times\R\times[0,1]}
  \W_p^p(\SC_{x,w},\SC_{F_{\mu}^{-1}(v),u})\chi(dv,du)\eta^k_{F_{\mu^k}^{-1}(v),u}(dx,dw)\end{align*}
The second term in the right-hand side goes to $0$ according to Lemma \ref{lem:eder} since $\bar\mu_\chi$ is the image of $\chi$ by $[0,1]\times[0,1]\ni(v,u)\mapsto (F_\mu^{-1}(v),u)\in\R\times [0,1]$ and, using $|x-F_{\mu}^{-1}(v)|^p\le 2^{p-1}(|x-F_{\mu^k}^{-1}(v)|^p+|F_{\mu^k}^{-1}(v)-F_{\mu}^{-1}(v)|^p)$, we have
\begin{align*}
  \int_{[0,1]^2\times\R\times[0,1]}|x-F_{\mu}^{-1}(v)|^p&+|w-u|^p\chi(dv,du)\eta^k_{F_{\mu^k}^{-1}(v),u}(dx,dw)\le 2^{p-1}\left(\AW_p^p(\SC^k,\SC)+\W_p^p(\mu^k,\mu)\right)\rightarrow 0.
\end{align*}
Hence $\int_{[0,1]\times[0,1]}\W_p^p\left(\SC^k_{F_{\mu^k}^{-1}(v),u},\SC_{F_{\mu}^{-1}(v),u}\right)\chi(dv,du)\rightarrow 0$.
Let $\pi^k$ (resp. $\pi$) denote the shadow coupling with copula $\chi$ between $\mu^k$ and $\nu^k$ (resp. $\mu$ and $\nu$) and for $(x,w)\in\R\times[0,1]$, $\vartheta^k(v,w)=F_{\mu^k}(F_{\mu^k}^{-1}(v)-)+w\mu^k(\{F_{\mu^k}^{-1}(v)\})$. The image of the Lebesgue measure on $[0,1]\times[0,1]$ by $\vartheta^k$ is the Lebesgue measure on $[0,1]$ and for each $v\in(0,1)$, $F_{\mu^k}^{-1}(\vartheta^k(v,w))=F_{\mu^k}^{-1}(v),\;dw$ a.e.. Hence $dv$ a.e.,
$$\pi^k_{F_{\mu^k}^{-1}(v)}=\int_{[0,1]\times[0,1]}\SC^k_{F_{\mu^k}^{-1}(\vartheta^k(v,w)),u}\chi_{\vartheta^k(v,w)}(du)dw.$$
Since $\mu$ does not weight points, $F_\mu^{-1}$ is one-to-one and $\pi_{F_{\mu}^{-1}(v)}=\int_{[0,1]}\SC_{F_{\mu}^{-1}(v),u}\chi_v(du)$, $dv$ a.e.. 
By the triangle inequality and Jensen's inequality (see for instance \cite[Proposition A.9]{BeJoMaPa21b}), we have
\begin{align*}
   \W_p(\pi^k_{F_{\mu^k}^{-1}(v)},\pi_{F_{\mu}^{-1}(v)})&\le \int_{[0,1]\times[0,1]}\W_p\left(\SC^k_{F_{\mu^k}^{-1}(\vartheta^k(v,w)),u},\SC_{F_{\mu}^{-1}(\vartheta^k(v,w)),u}\right)\chi_{\vartheta^k(v,w)}(du)dw\\&+\int_{[0,1]}\W_p\left(\int_{[0,1]}\SC_{F_{\mu}^{-1}(\vartheta^k(v,w)),u}\chi_{\vartheta^k(v,w)}(du),\int_{[0,1]}\SC_{F_{\mu}^{-1}(v),u}\chi_v(du)\right)dw.
\end{align*}
Using again that the image of the Lebesgue measure on $[0,1]\times[0,1]$ by $\vartheta^k$ is the Lebesgue measure on $[0,1]$, we deduce that \begin{align*}
  \AW_p^p(\pi,\pi^k)&\le\int_{[0,1]}|F_{\mu}^{-1}(v)-F_{\mu^k}^{-1}(v)|^p+\W_p^p(\pi_{F_{\mu}^{-1}(v)},\pi^k_{F_{\mu^k}^{-1}(v)})dv\\
  &\le \W_p^p(\mu,\mu^k)+2^{p-1}\int_{[0,1]\times[0,1]}\W_p^p\left(\SC^k_{F_{\mu^k}^{-1}(v),u},\SC_{F_{\mu}^{-1}(v),u}\right)\chi(dv,du)\\&+2^{p-1}\int_{[0,1]\times[0,1]}\W^p_p\left(\int_{[0,1]}\SC_{F_{\mu}^{-1}(\vartheta^k(v,w)),u}\chi_{\vartheta^k(v,w)}(du),\int_{[0,1]}\SC_{F_{\mu}^{-1}(v),u}\chi_v(du)\right)dwdv.\end{align*}
  The sum of the first two terms in the right-hand side goes to $0$ as $n\to\infty$.
Since, by the proof of \cite[Proposition 4.2]{JoMa22} (see the equation just above (4.12) where $\theta(F_\mu^{-1}(v),w)=v$ since $F_\mu$ is continuous), $dvdw$ a.e., $\vartheta^k(v,w)\rightarrow v$, we have $\int_{[0,1]\times[0,1]}|\vartheta^k(v,w)-v|^pdvdw\rightarrow 0$ by Lebesgue's theorem  so that the third term in the right-hand side also goes to $0$ by Lemma \ref{lem:eder} due to Eder \cite{Ed19}.
\end{proof}
\begin{remark}
Like in the proof of \cite[Proposition 4.2]{JoMa22}, we could check that $\AW_p(\pi^k,\pi)$ still goes to $0$ as $n\to\infty$ when
$$\forall x\in\R,\;\mu^k(\{x\})>0\Rightarrow\exists (x^k)_k\in\R^\N,\;F_{\mu^k}(x^k)\wedge F_\mu(x)-F_{\mu^k}(x^k-)\vee F_\mu(x-){\rightarrow} \mu(\{x\}).$$\end{remark}

\begin{proof}[Proof of Lemma \ref{lem:MZ}]
As $\theta^k \to \theta$ in total variation, we have that the total variation distance between $(\id,\varphi^k)_\# \theta^k$ and $(\id,\varphi^k)_\# \theta$ vanishes as $k \to \infty$.
Thus, since $((\id,\varphi^k)_\# \theta^k)_{k \in \N}$ converges to $(\id,\varphi)_\# \theta =: \eta$ in $\mathcal P({\cal V} \times {\cal Z})$, the same holds for the sequence $(\eta^k)_{k \in \N}$ where $\eta^k := (\id,\varphi^k)_\# \theta$.
 W.l.o.g.\ we assume that the metrics $d_\X$ and $d_\Y$ are both bounded, so that $\eta^k \to \eta$ in $\W_1$ and can pick couplings $\chi^k\in\Pi(\theta,\theta)$ such that
 \begin{align*}
   \W_1(\eta^k,\eta)=\int_{{\cal V}\times {\cal V}} d_{\cal V}(v,\hat v)+d_{\cal Z}(\varphi^k(v),\varphi(\hat v)) \, \chi^k(dv,d\hat v).
 \end{align*}
    By the triangle inequality we have
    \begin{align}\nonumber
        \int d_{\cal Z}(\varphi^k(v),\varphi(v)) \, \theta(dv) &= \int d_{\cal Z}(\varphi^k(v),\varphi(v)) \, \chi^k(dv,d\hat v) \\
        \nonumber
        &\le \int d_{\cal Z}(\varphi^k(v),\varphi(\hat v)) + d_{\cal V}(\varphi(\hat v), \varphi(v)) \, \chi^k(dv,d\hat v) \\
        \label{eq:lem.MZ.triangle}
        &= \W_1(\eta^k,\eta) +\int d_{\cal Z}(\varphi(v),\varphi(\hat v)) \, \chi^k(dv,d\hat v).
    \end{align}
    The first summand in \eqref{eq:lem.MZ.triangle} vanishes for $k \to \infty$ as $\eta^k \to \eta$ in $\mathcal W_1$, whereas the second summand vanishes as consequence of Lemma \ref{lem:eder} due to Eder \cite{Ed19}  since $\int d_{\cal V}(v,\hat v) \, \chi^k(dv,d\hat v) \to 0$.
\end{proof}

\begin{proof}[Proof of Lemma \ref{lem:w1bin}]
To show \ref{it:w1bin.1} and \ref{it:w1bin.2} we may assume w.l.o.g.\ that $y^k<x<z^k$, since $y^k=x=z^k$ implies $B(x,y^k,z^k)=\delta_x\neq B(x,y,z)$.
Then we compute
\begin{align}
    {\cal W}_1\left(B(x,y^k,z^k),B(x, y,z)\right)=&\left(\frac{z-x}{z-y}\wedge \frac{z^k-x}{z^k-y^k}\right)|y-y^k|+\left(\frac{z-x}{z-y}-\frac{z^k-x}{z^k-y^k}\right)^+(z^k-y)\notag\\&+\left(\frac{z^k-x}{z^k-y^k}-\frac{z-x}{z-y}\right)^+(z-y^k)+\left(\frac{x-y}{z-y}\wedge \frac{x-y^k}{z^k-y^k}\right)|z-z^k|.\label{eq:w1nbern}
\end{align}
When $|y^k-y|+|z^k-z|\to 0$ then each summand in \eqref{eq:w1nbern} also vanishes, showing the $\impliedby$-implication in \ref{it:w1bin.1}.
Conversely, when ${\cal W}_1(B(x,y^k,z^k), B(x,y,z)) \to 0$ then all four summands in the right-hand side of \eqref{eq:w1nbern} have to go to $0$ individually as $k \to \infty$.
Thus, since $z^k-y\ge x-y > 0$ and $z-y^k\ge z-x > 0$, we find, due to the second and third terms in \eqref{eq:w1nbern}, that $\frac{z^k-x}{z^k-y^k}\to \frac{z-x}{z-y}$.
Then using that the first and fourth terms also have to converge to 0, we get $|y^k-y|+|z^k-z| \to 0$, which completes the proof of \ref{it:w1bin.1}.
         
On the other hand, we have
\[
    \frac{\W_1\left(B(x,y^k,z^k),\delta_x\right)}{2} = \frac{(z^k-x)(x-y^k)}{z^k-y^k}=\frac{(z^k-x)\vee(x-y^k)}{z^k-y^k}\big((z^k-x)\wedge(x-y^k)\big).
\]
Since $z^k-y^k \ge (z^k-x)\vee(x-y^k) \ge \frac12 (z^k-y^k)$, we deduce that 
\[
    1 \le \frac{{\cal W}_1\left(B(x,y^k,z^k),\delta_x\right)}{(z^k-x)\wedge(x-y^k)}\le 2,
\]
which yields \ref{it:w1bin.2} and completes the proof.
\end{proof}

\subsection{Proof of Proposition \ref{prop:approximation_finite_pairs}}
\label{ssec:proof.prop.approximation_finite_pairs}

The proof of Proposition \ref{prop:approximation_finite_pairs} relies on Lemma \ref{lem:mathcal I approximates} below and Wasserstein projections in the convex order.
By \cite[Proposition 4.2]{AlCoJo20} there is a map $\Jast \colon \mathcal P_1(\R) \times \mathcal P_1(\R) \to \mathcal P_1(\R)$ satisfying
\begin{equation}
   \label{eq:wproj}
   \mu\le_c\Jast(\mu,\nu)\mbox{ and }{\cal W}_1(\Jast(\mu,\nu),\nu)=\inf_{\mu\le_c\eta}{\cal W}_1(\eta,\nu),
\end{equation}
which is called a Wasserstein projection in the convex order.
According to Theorem 1.1 \cite{JoMaPa22}, $\Jast$ is Lipschitz continuous: for  $\mu,\nu,\mu',\nu'\in\mathcal P_p(\R)$ with $p\ge 1$, we have
	\begin{align}
		\label{eq:J_Lipschitz}
		\mathcal W_p(\Jast(\mu,\nu),\Jast(\mu',\nu'))\le\phantom{2}\mathcal W_p(\mu,\mu')+2\mathcal W_p(\nu,\nu').
	\end{align}

\begin{lemma}\label{lem:mathcal I approximates}
	Let $a,b\in\R\cup\{-\infty,+\infty\}$, $a<b$, and $\rho \in \mathcal P_1(\R)$ be concentrated on $[a,b]$ with mean $x \in \R$.
	Let $(a^m)_{m\in\N},(b^m)_{m\in\N}$, $a < a^m < x < b^m < b$, be monotone sequences with $a^m \to a$ and $b^m \to b$.
    Then
    \[ \mathcal W_1\left(\rho\wedge_c\left(\frac{b^m - x}{b^m - a^m} \delta_{a^m} + \frac{x - a^m}{b^m - a^m} \delta_{b^m}\right),\rho\right) \to 0. \]
\end{lemma}

\begin{proof}[Proof of Lemma \ref{lem:mathcal I approximates}]
Let for each $m \in \N$, $\eta^m=\frac{b^m - x}{b^m - a^m} \delta_{a^m} + \frac{x - a^m}{b^m - a^m} \delta_{b^m}$ and $\rho^m=\rho\wedge_{c} \eta^m$. Let us check that $\lim_{m\to\infty}\sup_{y\in(-\infty,a^m]}(u_\rho(y)-(x-y))=0$. If $a=-\infty$, this is a consequence of $\lim_{y\to-\infty}(u_\rho(y)-(x-y))=0$. If $a>-\infty$, $u_\rho(y)=x-y$ for $y\le a$ and since $u_\rho$ is $1$-Lipschitz, $$\forall y\in[a,a^m],\;u_\rho(y)\le u_\rho(a)+(y-a)=x-a+y-a\le x-y+2(a^m-a).$$

        Since $u_{\eta^m}(y)=x-y$ for $y\in(-\infty,a^m]$ and using a symmetric reasoning  to deal with the supremum over $[b^m,+\infty)$ we deduce that $\lim_{m\to\infty}\sup_{y\in(-\infty,a^m]\cup[b^m,+\infty)}(u_\rho(y)-u_{\eta^m}(y))=0$. By convexity of $u_\rho$ and since $u_{\eta^m}$ is affine on $[a^m,b^m]$, $\sup_{y\in\R}(u_\rho(y)-u_{\eta^m}(y))=\sup_{y\in(-\infty,a^m]\cup[b^m,+\infty)}(u_\rho(y)-u_{\eta^m}(y))$ so that $$\lim_{m\to\infty}\sup_{y\in\R}(u_\rho(y)-u_{\eta^m}(y))=0.$$
        Since the convex function $u_\rho-\sup_{y\in\R}(u_\rho(y)-u_{\eta^m}(y))$ is not greater than $u_\rho\wedge u_{\eta^m}$, $$u_\rho-\sup_{y\in\R}(u_\rho(y)-u_{\eta^m}(y))\le \co(    u_\rho\wedge u_{\eta^m})=u_{\rho^m}\le u_\rho.$$
Hence $u_{\rho^m}$ converges uniformly to $u_\rho$ as $m\to\infty$, which implies that 
$\mathcal W_1(\rho^m,\rho)\underset{m\to+\infty}\longrightarrow0$.\end{proof}

\begin{proof}[Proof of Proposition \ref{prop:approximation_finite_pairs}]
    For every $j = 1,\ldots,J$, we have $\mu^k_j \to \mu_j$ in $\mathcal M_1(\R)$ with $\mu_j(\R) > 0$, so that $\mu^k_j(\R) \to \mu_j(\R) > 0$ and, for $k$ large enough, $\min_{1\le j\le J}\mu^k_j(\R)>0$.  
We are going to check the existence of $M<\infty $ such that for each $\varepsilon\in(0,1)$, we can find sequences $(\nu^{k,\varepsilon}_j)_{k\in\N}$ in $\mathcal M_1(\R)$ that satisfy
	\begin{equation}\label{summarynukj}
	\mu^k_j \leq_c \nu^{k,\varepsilon}_j, \quad\sum_{j = 1}^J \nu^{k,\varepsilon}_j = \nu^k\quad\text{and}\quad\limsup_{k\to+\infty}\sum_{j=1}^J\mu_j(\R)\mathcal W_1\left(\frac{\nu^{k,\varepsilon}_j}{\mu^k_j(\R)},\frac{\nu_j}{\mu_j(\R)}\right)\le M\varepsilon .
	\end{equation}
As $\varepsilon$ is arbitrary, the conclusion follows easily from \eqref{summarynukj}.

Before jumping into the various steps of proving \eqref{summarynukj}, we fix the following notation:
Let $a\in\{-\infty\}\cup\R$ and $b\in\R\cup\{+\infty\}$ be the endpoints of the irreducible component $I=(a,b)$ of $(\mu,\nu)$.
Further, let 
\[ 
    \pi^j = \frac{\mu_j}{\mu_j(\R)} \times \pi^j_x \in \Pi_M\left(\frac{\mu_j}{\mu_j(\R)},\frac{\nu_j}{\nu_j(\R)}\right). 
\]
Up to modifying $x\mapsto\pi^j_x$ on a $\mu$-null set, we suppose w.l.o.g.\ that for all $x\in(a,b)$, $\pi^j_x$ is concentrated on $[a,b]$ and $\mean(\pi^j_x)=x$. 
Finally, for $m \in \N$, pick $a^m,b^m \in I, a^m<b^m$, with $a^m \searrow a$, and $b^m \nearrow b$, so that $\mu_j([a^m,b^m]) > 0$ and $\mu_j(\{a^m,b^m\}) = 0$ for each $j = 1,\ldots,J$.

{\bf Step 1:} We claim that when $m$ is sufficiently large, there exists $\tilde\nu_j \in \mathcal M_1(\R)$ with
\begin{equation}
    \label{eq:step1}
    \W_1(\tilde\nu_j,\nu_j) < \epsilon,\quad
    \tilde\nu_j\le_c\nu_j,\quad
    \mu_j\vert_{[a^{m},b^{m}]}\le_c\tilde\nu_j\vert_{[a^{m},b^{m}]}\mbox{ and }    \tilde\nu_j\vert_{\R\backslash[a^m,b^m]}=\mu_j\vert_{\R\backslash[a^m,b^m]}.
\end{equation}
To show \eqref{eq:step1} we define $q_x^m$ as the unique probability measure supported on $\{a^m,b^m\}$ with $\mean(q_x^m) = x$ when $x \in [a^m,b^m]$, and $\delta_x$ otherwise, i.e.,
\[	
    q_x^m := \begin{cases} \frac{b^m - x}{b^m - a^m} \delta_{a^m} + \frac{x - a^m}{b^m - a^m} \delta_{b^m} & \text{if }x \in (a^m,b^m), \\ \delta_x & \text{else}.\end{cases}	
\]
Set $\pi^{j,m}(dx,dy) := \mu_j(dx)\,(\pi^{j}_x\wedge_c q^m_x)(dy)$. The measure $\pi^{j,m}$ is a martingale coupling between $\mu_j$ and its second marginal, which we denote by $\nu_{j,m}$ and thus $\nu_{j,m} \leq_c \nu_j$.
Thanks to Lemma \ref{lem:mathcal I approximates} we have for every $x \in (a,b)$ that $\mathcal W_1(\pi^j_x,\pi^{j}_x\wedge_c q^m_x)\to0$.
Furthermore, by the triangle inequality and convexity of the absolute value we have
\[
	\mathcal W_1(\pi^j_x, \pi^{j}_x\wedge_c q^m_x) \leq\mathcal W_1(\pi^j_x,\delta_0)+\mathcal W_1(\delta_0,\pi^{j}_x\wedge_c q^m_x)\le2\mathcal W_1(\pi^j_x,\delta_0),
\] 
where the right-hand side is $\mu_j$-integrable. 
Hence, we get by dominated convergence
 \begin{equation}\label{AW1pijpijm}
	\mathcal{AW}_1(\pi^j,\pi^{j,m}) \leq \int_\R \mathcal W_1(\pi^j_x,\pi^{j}_x\wedge_c q^m_x) \, \mu_j(dx)\underset{m\to+\infty}{\longrightarrow} 0.
\end{equation}
Letting $m$ be sufficiently large, \eqref{AW1pijpijm} yields that $\tilde\nu_j:=\nu_{j,m}$ satisfies $\W_1(\tilde\nu_j,\nu_j) < \varepsilon$.
Since, for $x\in[a^m,b^m]$, $(\pi^j_x\wedge_c q^m_x)([a^m,b^m])=1$ and for $x\in\R\backslash[a^m,b^m]$, $\pi^j_x\wedge_c q^m_x=\delta_x$, $\tilde\nu_j\vert_{[a^m,b^m]}$ is the second marginal of $\mu_j\vert_{[a^m,b^m]}\times(\pi^j_x\wedge_c q^m_x)$ and \eqref{eq:step1} holds.
Observe that $\tilde \nu_j([a^m,b^m]) = \mu_j([a^m,b^m])$ implies that 
\[
    \mean\left(\frac{\mu_j|_{[a^m,b^m]}}{\mu_j([a^m,b^m])} \right) =  \mean\left(\frac{\nu_j|_{[a^m,b^m]}}{\mu_j([a^m,b^m])} \right) =: \tilde x^m_j.
\]
	
{\bf Step 2:} Next we construct, for $j \in \{1,\ldots,J\}$, sequences $(\tilde \nu^k_j)_{k \in \N}$ in $\mathcal M_1(\R)$ such that
\begin{equation} \label{eq:step2}
    \mu^k_j \le_c \tilde \nu^k_j, \quad\sum_{j = 1}^J \tilde \nu^k_j \le_c \nu^k\quad\text{and }  \tilde\nu^k_j \underset{k\to\infty}{\to} (1-\epsilon)\tilde \nu_j + \epsilon\mu_j \text{ in }\mathcal M_1(\R).
\end{equation}
Since $\mu_j(\{a^m,b^m\}) = 0$, we have for every $h \in C_b(\R)$ that the discontinuities of $h \mathbbm 1_{[a^m,b^m]}$ are a $\mu_j$-null set, whence  we get by Portmanteau's theorem 
\[
    \int \mathbbm 1_{[a^m,b^m]}(x) h(x) \, \mu_j^k(dx) \to \int \mathbbm 1_{[a^m,b^m]}(x) h(x) \, \mu_j(dx).
\]
With Lemma \ref{lem:convergence_dominated}, we deduce that $\mu_j^k|_{[a^m,b^m]}$ converges to $\mu_j|_{[a^m,b^m]}$ in $\mathcal M_1(\R)$ as $k\to\infty$.
When $k$ is sufficiently large, we have $\mu_j^k([a^m,b^m]) > 0$ and define
\[
    x^{k,m}_j := \mean\left( \frac{\mu^k_j|_{[a^m,b^m]}}{\mu^k_j([a^m,b^m])} \right)\quad \text{and}\quad
    \hat \nu_j^k := \mathcal J\left( \frac{\mu_j^k\vert_{[a^{m},b^{m}]}}{\mu^k_j([a^{m},b^{m}])},
    \frac{\tilde \nu_j\vert_{[a^{m},b^{m}]}}{\mu_j([a^{m},b^{m}])}  \right) \wedge_{c} q_{x^{k,m}_j}^m.
\]
To simplify notation, we use the above definitions also when $\mu_j^k([a^m,b^m]) = 0$ under the convention that the undefined term $\frac{\mu^k_j|_{[a^m,b^m]}}{\mu^k_j([a^m,b^m])}$ is replaced by $\frac{\mu_j|_{[a^m,b^m]}}{\mu_j([a^m,b^m])}$.
By Lemma \ref{lem:convergence_subprobabilities} and $\mathcal W_1$-Lipschitz continuity of $\Jast$, c.f.\ \eqref{eq:J_Lipschitz}, we have
	\[
	\mathcal J	\left( \frac{\mu_j^k\vert_{[a^{m},b^{m}]}}{\mu_j^k([a^{m},b^{m}])},
	\frac{\tilde \nu_j\vert_{[a^{m},b^{m}]}}{\mu_j([a^{m},b^{m}])}  \right)\underset{k\to+\infty}{\longrightarrow} \mathcal J \left( \frac{\mu_j\vert_{[a^{m},b^{m}]}}{\mu_j([a^{m},b^{m}])},
	\frac{\tilde \nu_j\vert_{[a^{m},b^{m}]}}{\mu_j([a^{m},b^{m}])}  \right) = \frac{\tilde\nu_j\vert_{[a^m,b^m]}}{\mu_j([a^m,b^m])}\quad \text{in }\mathcal W_1,
	\]
 where the last equality follows from the fact that  $\mu\vert_{[a^{m},b^{m}]}\le_c\tilde\nu_j\vert_{[a^{m},b^{m}]}$.
Again by Lemma \ref{lem:convergence_subprobabilities}, we obtain that $x^{k,m}_j \to \tilde x_j^m$ and therefore $q_{x^{k,m}_j}^m \to q_{\tilde x^m_j}^m$ in $\mathcal W_1$.
Thus, \cite[Lemma 4.1]{BJMP22}, i.e.\ continuity of $\wedge_c$, provides that
\[
    \hat \nu^k_j \underset{k \to +\infty}{\longrightarrow} \frac{\tilde \nu_j|_{[a^m,b^m]}}{\mu_j([a^m,b^m])} \wedge_c q^m_{\tilde x_j^m} \quad \text{in }\W_1.
\]
Since $\tilde \nu_j\vert_{[a^m,b^m]}$ is concentrated on $[a^m,b^m]$ with mass $\mu_j([a^m,b^m])$ and mean $\tilde x_j^m$, we have $\tilde \nu_j\vert_{[a^m,b^m]}\leq_c \mu_j([a^m,b^m]) q^m_{\tilde x^m_j}$.
Hence, 
\begin{equation}\label{convhatnujk}
    \hat \nu^k_j \underset{k \to +\infty}{\longrightarrow} \frac{\tilde\nu_j\vert_{[a^m,b^m]}}{\mu_j([a^m,b^m])}\quad \text{in }\mathcal W_1.	
\end{equation}	
We set
\begin{equation}\label{deftildenukj}
    \tilde \nu^k_j := (1 - \epsilon) \left( \mu^k_j\vert_{\R \setminus [a^m,b^m]} + \mu^k_j([a^m,b^m]) \hat \nu^k_j \right) + \epsilon \mu^k_j.
\end{equation}
By definition of $\hat \nu^k_j$ and since $\frac{\mu^k_j|_{[a^m,b^m]}}{\mu^k_j([a^m,b^m])}\le_c q_{x^{k,m}_j}^m$, we have $\frac{\mu^k_j|_{[a^m,b^m]}}{\mu^k_j([a^m,b^m])}\le_c \hat\nu^k_j$, which yields
\begin{equation}\label{mukdominatedbynukj}
\mu^k_j=(1 - \epsilon) \left( \mu^k_j\vert_{\R \setminus [a^m,b^m]} + \mu^k_j\vert_{[a^m,b^m]}\right) + \epsilon \mu^k_j\le_c\tilde \nu^k_j.
\end{equation}
In order to complete step 2, it remains to show that, when $k$ is sufficiently large, then
\begin{equation}
    \label{eq:step2.2}
    \sum_{j = 1}^J \tilde \nu^k_j \le_{c} \nu^k.
\end{equation}
As $\mu^k_j \to \mu_j$ and $\mu^k_j|_{[a^m,b^m]} \to \mu_j|_{[a^m,b^m]}$ in $\mathcal M_1(\R)$ and by \eqref{convhatnujk} and \eqref{eq:step1}, we also obtain that $\tilde \nu^k_j \to (1 - \epsilon) \tilde \nu_j + \epsilon \mu_j$ in $\mathcal M_1(\R)$.
In turn, this implies locally uniform convergence of the sequence of potential functions $(u_{\tilde \nu^k_j})_{k \in \N}$ to $u_{(1-\epsilon)\tilde\nu_j + \epsilon\mu_j}$.
At the same time, as $\nu^k \to \nu$ in $\mathcal P_1(\R)$, we have uniform convergence of $u_{\nu^k}$ to $u_\nu$.
Thus, we find, for each $\delta > 0$, an index $k(\delta)\in \N$ such that for all $k\geq k(\delta)$ and $j\in\{1,\cdots,J\}$ we have
	\begin{equation}\label{potentialfunctionskgekdelta}
	u_{\tilde \nu^k_j}\vert_{[a^m,b^m]} \leq (1 - \epsilon) u_{\tilde \nu_j}\vert_{[a^m,b^m]} + \epsilon u_{\mu_j}\vert_{[a^m,b^m]} + \frac{\delta}{J}\quad\text{and}\quad
	u_{\nu} \leq u_{\nu^k} + \delta.
	\end{equation}
As $(\mu,\nu)$ is irreducible with component $I = (a,b) \supseteq [a^m,b^m]$, we can fix a $\delta > 0$ such that
	\begin{equation}\label{umuunudelta}
	\epsilon u_\mu\vert_{[a^m,b^m]} \leq \epsilon u_\nu\vert_{[a^m,b^m]} - 2\delta.
	\end{equation}
	Let $k \geq k(\delta)$, and compute, for $y \in [a^m,b^m]$, 
	\begin{align*}
	\sum_{j = 1}^J u_{\tilde \nu^k_j}(y) &\leq (1 - \epsilon) 
	\sum_{j = 1}^Ju_{\tilde \nu_j}(y) + \epsilon \sum_{j = 1}^Ju_{\mu_j}(y) + \delta
    \\
    &\leq (1 - \epsilon) u_\nu(y) + \epsilon u_\nu(y) - \delta
    \\
	&= u_\nu(y) - \delta\leq u_{\nu^k}(y),
	\end{align*}
    where the first and last inequalities follow from \eqref{potentialfunctionskgekdelta}, the second from $\sum_{j=1}^J\mu_j=\mu$, $\sum_{j=1}^J\tilde\nu_j\le_c\sum_{j=1}^J\nu_j=\nu$ and \eqref{umuunudelta}.
    Next, let $y \in \R \setminus [a^m,b^m]$.
    Since $\mu_j^k([a^m,b^m]) \hat \nu^k_j$ and $\mu^k_j|_{[a^m,b^m]}$ are both concentrated on $[a^m,b^m]$ with the same mass and barycentre, we obtain that their potential functions take the same value at $y$.
    We have
	\begin{align*}
	\sum_{j = 1}^J u_{\tilde \nu^k_j}(y)&=(1 - \epsilon)\sum_{j=1}^J\left( u_{\mu^k_j\vert_{\R \setminus [a^m,b^m]}}(y) + \mu^k_j([a^m,b^m])u_{\hat\nu^k_j}(y) \right) + \epsilon \sum_{j=1}^Ju_{\mu^k_j}(y)\\
	&= (1 - \epsilon)\sum_{j=1}^J\left( u_{\mu^k_j\vert_{\R \setminus [a^m,b^m]}}(y) + u_{\mu^k_j|_{[a^m,b^m]}}(y) \right) + \epsilon u_{\mu^k}(y)= u_{\mu^k}(y)\le u_{\nu^k}(y).
	\end{align*}
    Summarizing, we have $\sum_{j = 1}^J u_{\tilde \nu^k_j} \le u_{\nu^k}$ for $k \ge k(\delta)$, which yields \eqref{eq:step2.2}.

{\bf Step 3:} The final step consists in modifying $(\tilde \nu^k_j)_{j = 1}^J$ to $(\nu^{k,\epsilon}_j)_{j = 1}^J$ that fulfills \eqref{summarynukj}.
Denote by $\chi^k \in \Pi_M(\sum_{j = 1}^J \tilde \nu^k_j, \nu^k)$ the inverse transform martingale coupling (see \cite{JoMa20}), which by \cite[Theorem 2.11]{JoMa20} satisfies
	\begin{equation}\label{stabilityInequalityLambdaConvergence}
	\int_{\R\times\R}\vert y-x\vert\,\chi^k(dx,dy)\le2\mathcal W_1\left(\sum_{j=1}^J\tilde\nu^k_j,\nu^k\right).
	\end{equation}	
    We define $\nu^{k,\varepsilon}_j$ as the second marginal of $\tilde\nu^k_j\times\chi^k_x$, that is
	\[
	\nu^{k,\varepsilon}_j(dy) :=  \int_\R \chi^k_x(dy) \, \tilde\nu^k_j(dx).
	\]
    Using \eqref{mukdominatedbynukj} we have
	\begin{equation}\label{sumofnukjequelnuk}
	\mu^k_j\le_c\tilde\nu^k_j\le_c\nu^{k,\varepsilon}_j\quad\text{and}\quad\sum_{j=1}^J\nu^{k,\varepsilon}_j=\nu^k.
      \end{equation}
    To prove the remaining claim in \eqref{summarynukj}, we estimate
	\begin{align}
	\mu_j(\R)\mathcal W_1\left(\frac{\nu^{k,\varepsilon}_j}{\mu^{k}_j(\R)},\frac{\nu_j}{\mu_j(\R)}\right) &\leq \frac{\mu_j(\R)}{\mu^k_j(\R)}\mathcal W_1\left(\nu^{k,\varepsilon}_j,\tilde \nu^k_j\right) +\mathcal W_1\left(\frac{\mu_j(\R)}{\mu^k_j(\R)}\tilde \nu^k_j,\tilde \nu_j\right) + \mathcal W_1\left(\tilde \nu_j,\nu_j\right).\label{eqtriangw1}
	\end{align}
    Because of \eqref{eq:step2}, $\sum_{j=1}^J\tilde\nu^k_j$ converges to $(1-\varepsilon)\sum_{j=1}^J\tilde\nu_j+\varepsilon \sum_{j=1}^J\mu_j$ in ${\cal W}_1$.
    We compute
        \begin{align*}
     \limsup_{k\to\infty}\sum_{j=1}^J\frac{\mu_j(\R)}{\mu^k_j(\R)}\mathcal W_1\left(\nu^{k,\varepsilon}_j,\tilde \nu^k_j\right)&=\limsup_{k\to\infty}\sum_{j=1}^J\mathcal W_1\left(\tilde \nu^k_j,\nu^{k,\varepsilon}_j\right)\le \limsup_{k \to +\infty}\int_{\R\times\R} |y-x|\, \chi^k(dx,dy)\\&\le 2\limsup_{k \to +\infty} \mathcal W_1\left( \sum_{j = 1}^J \tilde \nu^k_j, \nu^k \right)  \le 2 \mathcal W_1\left( (1-\varepsilon)\sum_{j=1}^J\tilde\nu_j+\varepsilon \sum_{j=1}^J\mu_j,\nu\right)\\&\le 2(1-\varepsilon)\sum_{j=1}^J\mathcal W_1(\tilde\nu_j,\nu_j)+2\varepsilon\mathcal W_1(\mu,\nu)\le 2\left(J+\mathcal W_1(\mu,\nu)\right)\varepsilon,
    \end{align*}
    where the first equality is due to $\mu_j^k(\R) \to \mu_j(\R)$, the first inequality because $\sum_{j = 1}^J \tilde \nu^k_j \times \chi^k_x = \chi^k$, the second due to \eqref{stabilityInequalityLambdaConvergence}, the second last by the convexity of the 1-Wasserstein distance and $\nu=\sum_{j = 1}^J \nu_j$, and the last by \eqref{eq:step1}.
    As $\varepsilon<1$, we obtain by \eqref{eq:step2} and convexity of the 1-Wasserstein distance that
    \begin{align*}
       \limsup_{k\to\infty}\mathcal W_1\left(\frac{\mu_j(\R)}{\mu^k_j(\R)}\tilde \nu^k_j,\tilde \nu_j\right)\le \varepsilon \mathcal W_1(\mu_j,\tilde\nu_j)\le \varepsilon\big(\mathcal W_1(\mu_j,\nu_j)+\varepsilon\big)\le \varepsilon\big(\mathcal W_1(\mu_j,\nu_j)+1\big).
    \end{align*}
    Plugging these estimates into \eqref{eqtriangw1} yields
\begin{align*}
   \limsup_{k\to\infty}\sum_{j=1}^J\mu_j(\R)\mathcal W_1\left(\frac{\nu^{k,\varepsilon}_j}{\mu^{k}_j(\R)},\frac{\nu_j}{\mu_j(\R)}\right) \le \left(4J+2\mathcal W_1(\mu,\nu)+\sum_{j=1}^J\mathcal W_1(\mu_j,\nu_j)\right)\varepsilon
\end{align*}
so that \eqref{summarynukj} holds with $M=4J+2\mathcal W_1(\mu,\nu)+\sum_{j=1}^J\mathcal W_1(\mu_j,\nu_j)$.
\end{proof}

\bibliographystyle{abbrv}
\bibliography{joint_biblio}
\end{document}